\documentclass[12pt]{amsart}

\usepackage{amsmath,amsfonts,amsthm,amsopn,cite,mathrsfs, amssymb, bbm}
\usepackage{epsfig,verbatim,color, tikz}
\usepackage{url}
\usepackage{graphicx,picture,calc,caption}
\usepackage[latin1]{inputenc}
\usepackage[T1]{fontenc}
\usepackage[english]{babel}

\newcommand{\R}{{\mathbb{R}}}

\newcommand{\Nst}{{\mathbb{N}}}
\newcommand{\LtR}{{L^2(\R)}}

\newcommand{\abs}[1]{\lvert#1\rvert}
\newcommand{\C}{{\mathbb{C}}}

\newcommand{\Biggnorm}[1]{\Biggl\|#1\Biggr\|}
 
\newcommand{\bigabs}[1]{\bigl\lvert#1\bigr\rvert} 
\newcommand{\biggabs}[1]{\biggl\lvert#1\biggr\rvert} 
\newcommand{\card}[1]{\lvert#1\rvert}

\newtheorem{theorem}{Theorem}[section]
\newtheorem{lemma}[theorem]{Lemma}

\newcommand{\be}{\begin{equation}}
\newcommand{\ee}{\end{equation}}

\def\ba#1\ea{\begin{align}#1\end{align}}

\def\mwlet{\psi}
\def\transl{\mathbf T}
\def\dil{\mathbf D}

\def\E{\mathbb{E}}
\def\ind{\mathbbm{1}}
\def\var{\operatorname{Var}}

\allowdisplaybreaks[4]

\title[Filtering with Wavelet Zeros and GAFs]{Filtering with Wavelet Zeros and Gaussian Analytic Functions}

\author[L. D. Abreu]{Luis Daniel Abreu}
\email{labreau@kfs.oeaw.ac.at}
\address{Acoustics Research Institute, Austrian Academy of Sciences,
Wohl\-le\-ben\-gasse 12-14, 1040 Vienna, Austria}
\author[A. Haimi]{Antti Haimi}
\email{ahaimi@kfs.oeaw.ac.at}
\address{Acoustics Research Institute, Austrian Academy of Sciences,
Wohl\-le\-ben\-gasse 12-14, 1040 Vienna, Austria}
\author[G. Koliander]{G\"{u}nther Koliander}
\email{gkoliander@kfs.oeaw.ac.at}
\address{Acoustics Research Institute, Austrian Academy of Sciences,
Wohl\-le\-ben\-gasse 12-14, 1040 Vienna, Austria}
\author[J. L. Romero]{Jos\'{e} Luis Romero}
\email{jose.luis.romero@univie.ac.at, jlromero@kfs.oeaw.ac.at}
\address{Faculty of Mathematics, University of Vienna,
Oskar-Morgenstern-Platz 1, 1090 Vienna, Austria\\
and\\
Acoustics Research Institute, Austrian Academy of Sciences, Wohl\-leben\-gasse
12-14, 1040 Vienna, Austria}
\thanks{This work was supported by the Austrian Science Fund
(FWF): P 31153-N35 and P 29462-N35 and the Vienna Science and Technology Fund (WWTF): MA16-053.}%

\begin{document}

\begin{abstract}
  We present the continuous wavelet transform (WT) of white Gaussian noise and establish a connection to the theory of Gaussian analytic functions. 
  Based on this connection, we propose a methodology that detects components of a signal in white noise based on the distribution of the zeros of its continuous WT. 
  To illustrate that the continuous theory can be employed in a discrete setting, we establish a uniform convergence result for the discretized continuous WT and apply the proposed method to a variety of acoustic signals.
\end{abstract}

\maketitle

%%%%%%%%%%%%%%%%%%%%%%%%%%%%%%%%%%%%%%%%%
\section{Introduction}
%%%%%%%%%%%%%%%%%%%%%%%%%%%%%%%%%%%%%%%%%

Identifying important components of signals embedded in a noisy background is a
fundamental problem in signal analysis.
Most methods in the literature aim to identify the signal from their large components in some transform domain
 (corresponding to high energy regions as in (block-)thresholding methods \cite{Shrinkage,yumaba08}, curves as in ``synchrosqueezing'' \cite{dama96} or ``reassignment'' \cite{aufl95} methods \cite{auflli13}, or ridges \cite{cahwto97,Ridges}). 
In sharp contrast,
Flandrin introduced a novel method to identify a signal embedded in Gaussian
white noise based on its ``silent points''---the \emph{zeros} of the
short-time Fourier transform (STFT) \cite{flandrin2015time}. 
Similarly,  the
``loud-silent'' dichotomy represented by maxima and minima of Gaussian
spectrograms based on observation of their geometric distribution was studied in 
\cite{FlandrinMaximum,Silence}. 
Flandrin's analysis relies on the representation of an STFT with a Gaussian window as an entire function (up to a non-vanishing factor).

In Flandrin's method, the spectrogram of the clean signal is singled out as an area of statistical deviation from the pattern expected from noise. The intuition behind the filtering procedure is that, while the zeros of the STFT of pure Gaussian white noise with a Gaussian window are distributed according to a very regular random pattern in the plane, the presence of a deterministic signal
perturbs that pattern. 
This insight has been recently revisited in
\cite{bardenet2017zeros}, by noting that the spectrogram of white noise with
respect to a Gaussian window is a symmetric Gaussian entire function
%---i.e., a random analytic function $\sum_{j}a_{j}\frac{z^{j}}{\sqrt{j!}}$,
%where $a_{j}$ are i.i.d.\ real Gaussian random variables---
and thus its
zero-set obeys well-known statistics \cite{feldheim2013zeroes}.

We propose a scheme similar to Flandrin's, but based on a
continuous \emph{wavelet} transform (WT) \cite[Ch.~2]{da92} with analyzing wavelets of the form
\begin{equation}
\label{eq_ga}
\widehat{\mwlet_{\alpha }}(\xi ):=
\begin{cases}
\xi^{\frac{\alpha
-1}{2}}e^{-\xi }, & \xi \geq 0, \\
0, & \xi <0,
\end{cases}
\end{equation}
with $\alpha >1$. 
The starting point of our analysis is the observation that
these windows lead to WTs that map (again up to a non-vanishing factor) into a  space of
analytic functions in the upper
half-plane \cite{dapa88}. 
As a
consequence, we identify the point process arising from the zeros of the
scalograms of  white noise with the zero set of a so-called \emph{hyperbolic Gaussian analytic function} (GAF), and use this information to
propose an adequate filtering procedure.

Similarly to the existing approach for the STFT case, we expect that large deviations from the zero pattern that is expected for white noise indicate signal components.
However, in contrast to 
\cite{flandrin2015time,Silence}, we do not rely on a triangulation to identify the corresponding regions of deviation.
Instead, we use estimated moment densities. 
More specifically, we calculate local approximations of the first intensity and the pair correlation function and compare them to the analytic form of these expressions for white noise derived from \cite{hough2009zeros}.
We expect that local deviations hint at signal components and propose a filtering procedure based on these deviations.

Our proposed methodology can also be applied to real signals that are not given as a continuous waveform but as a finite sequence of samples. 
We show that the discretized continuous WT of discrete white noise converge locally uniformly to the continuous WT, i.e., a hyperbolic GAF, with probability one.
Thus, the theory for the continuous setting can also be expected to hold approximately for the discrete case.
In simple experiments, we first show that for discrete white noise samples and the discrete CWT, we obtain statistics that remarkably well reflect the numbers proposed by continuous theory. 
To our knowledge, this is also the first faithful simulation of the hyperbolic GAF.
We further illustrate our methodology on acoustic signals that are superimposed with artificial white noise. 
We observe that although the theoretical results were derived only for complex white noise, the same behavior is observed for real white noise.

The paper is organized as follows. 
In Section \ref{sec:AIWT}, we define the analyticity inducing wavelet transform that will be used throughout the paper.
In Section \ref{sec_wn}, white noise is introduced, we extend the WT to white noise, and discuss the connection to Gaussian analytic functions.
The relation to the discrete domain is established in Section~\ref{sec:dwt}
where we prove that the discretized continuous WT of white noise converges to the  continuous WT of white noise.
In Section~\ref{sec:zerogaf}, we present basic properties of the zero set of a Gaussian analytic function, in particular, the first intensity function and the pair correlation function.
Local estimators for these functions based on a point pattern are then given in Sections~\ref{sec:estfif} and \ref{sec:estpcf}, respectively.
Finally, in Section~\ref{sec:experiments}, we present illustrative experiments before concluding the paper in Section~\ref{sec:conclusion}.

%%%%%%%%%%%%%%%%%%%%%%%%%%%%%%%%%%%%%%%%%
\section{Analyticity Inducing Wavelet Transform}\label{sec:AIWT}
%%%%%%%%%%%%%%%%%%%%%%%%%%%%%%%%%%%%%%%%%

Let $\mwlet\in\LtR$ such that its Fourier transform $\widehat{\mwlet}$ vanishes almost everywhere on 
  $\R^-$.  
  	The continuous WT  of a function (or signal) $s\in\LtR$ with respect to the \emph{mother wavelet} 
  $\mwlet$ is defined as
  \begin{equation}
    W_{\mwlet} s(x,y) = \langle s,\transl_x \dil_y \mwlet \rangle = \frac{1}{\sqrt{y}} \int_\R s(t)\overline{\mwlet\left(\frac{t-x}{y}\right)}\, dt, 
		\label{eq:defwltransform}
  \end{equation}
  for all
  %\footnote{Although we restrict here to positive scales, there is no technical obstruction to allowing $y\in\RR\setminus\{0\}$.} 
  $x\in\R$, $y\in\R^+$. 
  Here, $\transl_x$ and $\dil_y$ denote the translation and dilation operators, respectively,  given by 
   $(\transl_x s)(t) = s(t-x)$, and $(\dil_y s)(t) = y^{-1/2} s(t/ y)$ for all $t\in \R$.
  The \emph{admissibility constant} $C_{\mwlet}$ of a wavelet $\mwlet$ is defined as 
  \be
    C_{\mwlet} = \int_{\R^+} \frac{\abs{\widehat{\mwlet} (\xi) }^2}{ \xi} \, d\xi
  \ee
  and  $\mwlet$ is called admissible if $C_{\mwlet} < \infty$.

We are interested in mother wavelets $\mwlet$ such that the image of any function is (up to scaling) an analytic function in the complex variable $x+iy$.
The class of all $\mwlet$ satisfying this analyticity inducing property was recently characterized in \cite{hokopr19} and consists essentially of the Cauchy wavelets specified by \eqref{eq_ga} times a chirp (also known as ``Klauder wavelets'' \cite{fl98}).
More specifically, the function
\be
  x+iy  \mapsto y^{-\frac{\alpha}{2}} W_{\mwlet_{\alpha }} s(x,y)
\ee
is analytic for $\alpha>-1$ and $s\in \LtR$.
In the following, we consider the setting $\alpha>1$ for which the wavelets $\mwlet_{\alpha }$ are admissible.
In this case,  the operator 
%$s \mapsto y^{-1} W_{\mwlet_{\alpha }} s(x,y)$ is up to a constant factor an isometry from $\LtR$ to 
%$L^2(\R \times \R^+)$, or, equivalently, 
$s \mapsto y^{-\frac{\alpha}{2}} W_{\mwlet_{\alpha }} s(x,y)$ is up to a constant an isometric isomorphism from $\LtR$ to 
the space $A^2_{\alpha}$  of all analytic functions $f$ on
$\Pi^{+}$ such that the norm
\begin{equation}
\lVert f\rVert _{A_{\alpha }^{2}(\Pi ^{+})}^{2}=\frac{1}{\pi }\int_{\Pi
^{+}}\left\vert f(x+iy)\right\vert ^{2}y^{\alpha -2}\,\mathrm{d}m(z)
\end{equation}
is finite \cite{dapa88}. Here, $\mathrm{d}m(z)=\mathrm{d}x\mathrm{d}y$ denotes the Lebesgue measure.
The spaces $A^2_{\alpha}$ are known as weighted Bergman spaces on the upper halfplane and are Hilbert spaces of analytic functions with inner product induced by $\lVert \cdot\rVert _{A_{\alpha }^{2}(\Pi ^{+})}$.

%%%%%%%%%%%%%%%%%%%%%%%%%%%%%%%%%%%%%%%%%
\section{Wavelet Transform of White Noise}\label{sec_wn} 
%%%%%%%%%%%%%%%%%%%%%%%%%%%%%%%%%%%%%%%%%

As a first step in our analysis, we introduce a rigorous definition of white noise.
%In order to analyze the properties of the WT of white noise, we first introduce a rigorous definition of it.
Specifically, we adopt a Gaussian Hilbert space approach
\cite{janson1997gaussian}. 
Let $(\Omega ,\mathcal{F},\mathbb{P})$ be a
probability space. Heuristically, one thinks of white noise on ${\mathbb{R}}$
as a linear combination $\mathcal{N} =\sum_{n=0}^{\infty }a_{n}e_{n}$ where $a_{n}$
are independent standard (real or complex) Gaussians and $\{e_{n}:n\geq 0\}$
is an orthonormal basis of $\LtR$.
Unfortunately, this
sum does not converge in $\LtR$ with probability $1$ (the sequence $\{a_{n}\}_{n\in \mathbb{N}}\notin \ell^2$ with probability $1$ because $\sum_{n=0}^N\abs{a_n}^2$ follows a chi-square distribution with $N+1$ degrees of freedom and, thus, the probability that this sum is less than any finite constant decreases in $N$ towards 0).
However, for any $s\in \LtR$, the sum
\begin{equation}
\mathcal{N}(s):=\sum_{n= 0}^{\infty} a_{n}\langle s,e_{n}\rangle
\label{eq:noiseelements}
\end{equation}
converges in $L^{2}(\Omega ,\mathcal{F},\mathbb{P})$ to a complex Gaussian
variable with mean zero and variance $\Vert s\Vert ^{2}$. A precise
definition of white noise is then as the collection of random variables $G:=\{\mathcal{N}(s): s\in L^{2}(\mathbb{R})\}$. The space $G$ is a
\emph{Gaussian Hilbert space}, that is, a Hilbert space consisting of Gaussian
random variables. Its inner product is induced by
%\begin{equation*}
$\Vert \mathcal{N}(s)\Vert _{G}^{2}:=\Vert s\Vert ^{2}$.
%\end{equation*}
We will call the white noise real or complex depending on whether the
variables $a_{n}$ are real or complex standard Gaussians. 

%Alternatively, white noise could be defined as a random element in some
%larger class, e.g., tempered distributions, but, with such definition,
%$\mathcal{N}(s)$ would not be defined for all $s \in \LtR$.

%%%%%%%%%%%%%%%%%%%%%%%%%%%%%%%%%%%%%%%%%
%\section{Wavelet Transform of White Noise}
%%%%%%%%%%%%%%%%%%%%%%%%%%%%%%%%%%%%%%%%%

We next turn to the extension of the WT to white noise. 
We extend the WT with respect to the windows $\mwlet_{\alpha }$  to (real or
complex) white noise by
\begin{equation}
W_{\mwlet_{\alpha }}(\mathcal{N})(z)= \mathcal{N} (\transl_x \dil_y \mwlet_{\alpha })
=\sum_{n= 0}^{\infty} a_{n}\langle \transl_x \dil_y \mwlet_{\alpha },e_{n}\rangle
 =\sum_{n= 0}^{\infty}a_{n}W_{\mwlet_{\alpha }}e_{n}(z),
\label{eq:wavelet_wn}
\end{equation}
where $z=x+iy$. 
By the isometry property of the WT, there exists a constant $c_{\alpha }>0$ such that
 $\{b_n: n \geq 0\}$, with $b_n:=c_{\alpha } y^{-\frac{\alpha}{2}} W_{\mwlet_{\alpha }}(e_{n})$, is an orthonormal basis of  $A^2_{\alpha}$.
The series in \eqref{eq:wavelet_wn} can thus be rewritten as
\be
W_{\mwlet_{\alpha }}(\mathcal{N})(z)=\frac{y^{\frac{\alpha}{2}} }{c_{\alpha } }\sum_{n= 0}^{\infty}a_{n}b_{n}(z)
= \frac{y^{\frac{\alpha}{2}} }{c_{\alpha } } f_{\alpha}(z),
\label{eq:wavelet_wngaf}
\ee
where $f_{\alpha}:= \sum_{n=0}^{\infty }a_{n}b_{n}$ is a so-called hyperbolic Gaussian analytic function (GAF) \cite{hough2009zeros} on  $A^2_{\alpha}$.
We note that our setting of admissible wavelets (i.e., $\alpha>1$) does not include the determinantal point process case that would correspond to the case $\alpha=1$. 
Although the hyperbolic GAF is most commonly defined with the explicit basis%
\footnote{This corresponds to appropriately normalized monomials  when  mapped to the unit disk by the isomorphism $(\mathbf{T}^{\alpha }f)(z)=
\frac{2}{(1-z)^{\alpha }}f\big( i\frac{z+1}{1-z}\big)$.}
\be
  b_{n}(z) = 2^{\alpha -1}\sqrt{\frac{\Gamma (n+\alpha )}{n!\Gamma
(\alpha -1)}}\left( \frac{z-i}{z+i}\right) ^{n}\left( \frac{i}{z+i}\right)
^{\alpha } 
\ee
the statistical properties depend only on the Szeg\"o kernel \cite{elliott2011composition}
\be
  K(z,w) 
  = \sum_{n=0}^{\infty}b_n(z)\overline{b_n(w)}
  = \frac{2^{\alpha -2}(\alpha -1)}{(-i)^{\alpha }
(z-\bar{w})^{\alpha }}
\ee
and are thus independent of the chosen basis \cite{peres2005zeros}. 

Note that the  GAF $f_{\alpha}$ does not take values in $A^2_{\alpha}$ since, as explained above, $\{a_{n}\}_{n\in \mathbb{N}}\notin \ell^2$.
However, the defining series converges almost surely and locally uniformly to an analytic function \cite{hough2009zeros} and thus the zeros of a
realization are well defined.
We are interested in the distribution of zeros of the WT of white noise.
By  \eqref{eq:wavelet_wngaf}, we obtain that the set of zeros of
$W_{\mwlet_\alpha}(\mathcal{N})$, where $\mathcal{N}$ is complex white noise, has the same distribution as those of $f_\alpha$, the hyperbolic GAF associated with the Bergman space $A^2_\alpha$.
This observation was also recently made in \cite{baha18}.
We also point out that, while the characterization of analytic functions from their factorization in terms of zeros leave some multiplicative factors undetermined, 
the GAFs we are considering are completely determined by their zeros (see~\cite[Th.~6]{peres2005zeros}) with probability one.

%%%%%%%%%%%%%%%%%%%%%%%%%%%%%%%%%%%%%%%%%
\section{Discretization of White Noise}\label{sec:dwt}
%%%%%%%%%%%%%%%%%%%%%%%%%%%%%%%%%%%%%%%%%

One of the main motivations of this work is to use the theory for signals in  $\LtR$ and apply it to a signal that is given by samples on a finite interval.
In order to use the continuous theory, we thus provide a link with the finite discrete setting.

In a discrete setting, white noise is much less troublesome to define and is simply a sequence of i.i.d.\ Gaussian random variables.
To get a link to the continuous world, we define
a matching discretization of continuous white noise by the application of $\mathcal{N}$ to indicator functions on small intervals $((\ell-1)T_s, \ell T_s)$, where $T_s > 0$ is the sampling interval and $\ell \in \{-L+1, \dots, L\}$, i.e.,
\be
  \mathcal{N}_{L,T_s}[\ell] = \mathcal{N}(\ind_{((\ell-1)T_s, \ell T_s)})\,.
\ee
The resulting random variables $\mathcal{N}_{L,T_s}[\ell]$ are independent and identically distributed complex Gaussians with mean $0$ and variance $T_s$.
The discrete continuous wavelet transform of these noise samples is given by 
\ba 
  W^{(d)}_{\mwlet_{\alpha }}(\mathcal{N}_{L,T_s})(z)
  & = \sum_{\ell=-L+1}^L \mathcal{N}_{L,T_s}[\ell] \transl_x \dil_y \mwlet_{\alpha }(\ell T_s) 
  \notag \\
  & = \mathcal{N} \bigg( \sum_{\ell=-L+1}^L \transl_x \dil_y \mwlet_{\alpha }(\ell T_s)\ind_{((\ell-1)T_s, \ell T_s) } \bigg).
  \label{eq:wtdisc}
\ea
Our goal is to show that almost surely the random function $W^{(d)}_{\mwlet_{\alpha }}(\mathcal{N}_{L,T_s})$ converges for $L\to \infty$ to $W_{\mwlet_{\alpha }}(\mathcal{N})$ locally uniformly in $z$.
More specifically, we have the following result.
\begin{theorem}\label{th:approxdisc}
  For $T_s=L^{-\frac{\alpha}{\alpha+2}}$ and any compact $D\subseteq \Pi^+$
  \be
    \Pr\Big[ \lim_{L\to \infty} \sup_{z\in D} \, \bigabs{W^{(d)}_{\mwlet_{\alpha }}(\mathcal{N}_{L,T_s})(z) - W_{\mwlet_{\alpha }}(\mathcal{N})(z)} =0 \Big] =1 \,.
  \ee
\end{theorem}
\begin{proof}
  We will show that for any $\varepsilon>0$
  \be \label{eq:prfixedlsum}
    \sum_{L\in \Nst} \Pr\Big[  \sup_{z\in D} \, \bigabs{W^{(d)}_{\mwlet_{\alpha }}(\mathcal{N}_{L,T_s})(z) - W_{\mwlet_{\alpha }}(\mathcal{N})(z)} >\varepsilon \Big] < \infty\,.
  \ee
  The theorem then follows from the Borel-Cantelli lemma.

  Exceedance probabilities such as the ones in \eqref{eq:prfixedlsum} are well-studied for a Gaussian random function $F$ with continuous realizations \cite{adler07} (of which our GAFs are a special case) and  depend on bounds on the maximal variance in $D$ (denoted as $\sigma_D^2 = \sup_{z \in D} \E[\abs{F(z)}^2]$) and a bound on a specific covering number of $D$.
  More specifically, we have to show that there exists $A^2> \sigma_D^2$ such that
  \be \label{eq:covnr}
    N(D,d_L,\varepsilon) \leq \frac{A^2}{\varepsilon^2}
  \ee
  for all $\varepsilon\leq \sigma_D$, where
  $N(D,d_L,\varepsilon)$ denotes the smallest number of disks of radius $\varepsilon$ with respect to the metric $d_L$ that is required to cover $D$.
  Here, $d_L$ is defined as
  \be
    d_L(w,z) = \big( \E \big[ \abs{F(w)- F(z) }^2 \big] \big)^{1/2}
  \ee
  Provided the bound \eqref{eq:covnr} holds, 
  \cite[Th.~4.1.2]{adler07} implies that there exists a universal constant $K$ such that
  \be\label{eq:boundoneprob}
    \Pr\Big[  \sup_{z\in D} \, \abs{F(z)} >u \Big] 
    \leq
    \bigg(\frac{K A u}{\sqrt{2}\sigma_D^2}\bigg)^2 \Psi\bigg(\frac{u}{\sigma_D}\bigg)
  \ee
  for all $u \geq (1+\sqrt{2}) \sigma_D$ where $\Psi(x) = \frac{1}{\sqrt{2\pi}}\int_x^{\infty} e^{-\frac{t^2}{2}} \mathrm{d}t$.
  
  We first note that almost every realization of $W_{\mwlet_{\alpha }}(\mathcal{N})$ is continuous because it is analytic with probability one. 
  Furthermore, $W^{(d)}_{\mwlet_{\alpha }}(\mathcal{N}_{L,T_s})$ has  continuous realizations as it is in fact just a random finite sum of the continuous functions $(x,y)\mapsto \mwlet_{\alpha }\big(\frac{\ell T_s -x}{y}\big)$.
  The difference $F_L(z):=W^{(d)}_{\mwlet_{\alpha }}(\mathcal{N}_{L,T_s})(z) - W_{\mwlet_{\alpha }}(\mathcal{N})(z)$ is now a Gaussian random  function in $z$ with   continuous realizations.

  Thus, it remains to show  \eqref{eq:covnr} and to calculate the maximal variance $\sigma_D^2$.
  The bound on the variance is presented in Lemma~\ref{lem:variance} in Appendix~\ref{app:lem} based on the definition of the variance of $\mathcal N(f)$ and the Lipschitz continuity of $\mwlet_{\alpha}$ and gives us
  \be
    \sigma_D^2 
    \leq
    C_1 L^{-2\frac{\alpha-1}{\alpha+2}}
    \label{eq:upperbsigd}
  \ee
  for some constant $C_1$ that depends on $D$ and $\alpha$ but not on $L$.
  
  To obtain the bound \eqref{eq:covnr} we first establish the bound  $d_L(w,z)\leq C_2 L^{ \frac{3}{ \alpha+2 }} \abs{w-z}$ for all $w,z \in D$ and some constant $C_2$ that that depends on $D$ and $\alpha$ but not on $L$, $w$, or $z$ in Lemma~\ref{lem:bounddl} in Appendix~\ref{app:lem}.
  Because $4 x_{\textrm{max}} y_{\textrm{max}}/\tilde{\varepsilon}^2$ Euclidean disks of radius $\tilde{\varepsilon}$ suffice to cover a rectangle $[-x_{\textrm{max}}, x_{\textrm{max}}] \times [0, y_{\textrm{max}}]$ that is sufficiently large to cover $D$, we need at most $4 C_2^2 L^{ \frac{6}{ \alpha+2 }} x_{\textrm{max}} y_{\textrm{max}}/\varepsilon^2$ Euclidean disks of radius $\varepsilon/(C_2L^{ \frac{3}{ \alpha+2 }})$ to cover $D$.
  Since by Lemma~\ref{lem:bounddl} $d_L(w,z)\leq C_2 L^{ \frac{3}{ \alpha+2 }} \abs{w-z}$ (i.e., disks of $d_L$ radius $\varepsilon$ are larger than disks of Euclidean radius $\varepsilon/(C_2L^{ \frac{3}{ \alpha+2 }})$)
  this implies that \eqref{eq:covnr} holds with $A^2= 4 C_2^2 L^{ \frac{6}{ \alpha+2 }}x_{\textrm{max}} y_{\textrm{max}}$, i.e., 
  \be
    N(D,d_L,\varepsilon) \leq \frac{4 C_2^2 L^{ \frac{6}{ \alpha+2 }}x_{\textrm{max}} y_{\textrm{max}}}{\varepsilon^2}\,.
  \ee
  
  We finally have all the ingredients to use the bound \eqref{eq:boundoneprob}.
  For $\sigma_D \leq \frac{1}{3+2 \sqrt{2}}$, we choose $u = \sigma_D^{\frac{1}{2}}$ and obtain
  \be\label{eq:boundoneprob2}
    \Pr\Big[  \sup_{z\in D} \, \bigabs{W^{(d)}_{\mwlet_{\alpha }}(\mathcal{N}_{L,T_s})(z) - W_{\mwlet_{\alpha }}(\mathcal{N})(z)} >\sigma_D^{\frac{1}{2}} \Big] 
    \leq
    \frac{C_3 L^{ \frac{6}{ \alpha+2 }}}{ \sigma_D^3} \Psi\big( \sigma_D^{-\frac{1}{2}}\big)\,.
  \ee
  For $\sigma_D$ sufficiently small we further have that  $\sigma_D^{-3} \Psi\big( \sigma_D^{-\frac{1}{2}}\big)$ is monotonically increasing in $\sigma_D$ by Lemma~\ref{lem:psimonoton} in Appendix~\ref{app:lem} and thus we can upper bound the right-hand side of \eqref{eq:boundoneprob2} by replacing $\sigma_D$ with its upper bound \eqref{eq:upperbsigd}, i.e.,
  \ba
    & \Pr\Big[  \sup_{z\in D} \, \bigabs{W^{(d)}_{\mwlet_{\alpha }}(\mathcal{N}_{L,T_s})(z) - W_{\mwlet_{\alpha }}(\mathcal{N})(z)} >\sigma_D^{\frac{1}{2}} \Big] 
    \notag \\
    & \leq
    \frac{C_3 L^{ \frac{6}{ \alpha+2 }}}{C_1^{\frac{3}{2}}L^{-3\frac{\alpha-1}{\alpha+2}}} \Psi\big( C_1^{-\frac{1}{4}}  L^{ \frac{\alpha-1}{2\alpha+4}}\big)
    \notag \\
    & =
    C_1^{-\frac{3}{2}}C_3 L^{ \frac{3\alpha +3 }{ \alpha+2 }} \Psi\big( C_1^{-\frac{1}{4}} L^{ \frac{\alpha-1}{2\alpha+4}}\big)
    \notag \\
    & \leq
    C_4 L^{-2} 
    \label{eq:boundprobfinal}
  \ea
  where the final inequality follows because for $L$ sufficiently large we have that $L^2 L^{ \frac{3\alpha +3 }{ \alpha+2 }} \Psi\big( C_2 L^{ \frac{\alpha-1}{2\alpha+4}}\big)$ is monotonically decreasing by Lem\-ma~\ref{lem:psimonoton} in Appendix~\ref{app:lem}.
  As for $L$ sufficiently large $\sigma_D$ becomes arbitrarily small, we have in particular $\sigma_D^{\frac{1}{2}}<\varepsilon$ for all sufficiently large $L$.
  Thus, \eqref{eq:boundprobfinal} implies \eqref{eq:prfixedlsum} which concludes the proof.
\end{proof}

Theorem~\ref{th:approxdisc} motivates the use of the continuous theory also in a discrete setting. 
In the experiments in Section~\ref{sec:experiments} below,  we will also see that the statistics conform well to what is predicted by the continuous theory.

A second problem in the discrete setting is that the zeros of the scalogram are not on the sampling grid with probability one. 
To resolve this issue, we use Lemma~\ref{lem:locminzero} below stating that every local minimum of the scalogram is a zero. 
Thus, instead of finding zeros, we can look for local minima, which are robust under discretization.
\begin{lemma}\label{lem:locminzero}
  Let $W_{\mwlet_{\alpha}} s(x,y)$ be given by \eqref{eq:defwltransform} and \eqref{eq_ga}. 
  Every local minimum of $\abs{W_{\mwlet_{\alpha}} s}$ is a zero.
\end{lemma}
\begin{proof}
The proof is inspired by the STFT case (cf.~\cite[Sec.~8.2.2]{hough2009zeros}) but  circumvents the use of any gradient curves and does not rely on the factorization of analytic functions: 
Recall that the wavelet transform for Cauchy wavelets results in a weighted analytic function on $\Pi^+$, i.e.,  $W_{\mwlet_{\alpha}} s(x,y) = y^{\frac{\alpha}{2}} f(x+iy)$ where $f$ is analytic.
Taking the logarithm of the modulus, we have that 
\be
  \log \big(y^{\frac{\alpha}{2}} \abs{f(z)} \big)
  = \frac{\alpha}{2}\log y - \log \frac{1}{\abs{f(z)}}
\ee
where $1/f(z)$ is analytic on $\Pi^+$ except at the zeros of $f$ and in turn $\log \frac{1}{\abs{f(z)}}$ is subharmonic. 
Taking $\Delta\log y = -y^{-2}$, where $\Delta$ denotes the Laplace operator, we see that $\log y$ is superharmonic and thus $\frac{\alpha}{2}\log y - \log \frac{1}{\abs{f(z)}}$ is superharmonic and cannot have any local minima by the minimum principle on $\Pi^+$, except at the zeros of $f$.
\end{proof}

%%%%%%%%%%%%%%%%%%%%%%%%%%%%%%%%%%%%%%%%%
\section{The Zero Set of a Hyperbolic Gaussian Analytic Function}\label{sec:zerogaf}
%%%%%%%%%%%%%%%%%%%%%%%%%%%%%%%%%%%%%%%%%

The zero set of GAFs
is a well-studied point process.
We
can think of a point process as a random integer-valued measure, by setting a Dirac mass
at each zero. 
This point process is \emph{simple} \cite[Lem.~2.4.1]{hough2009zeros}, which means that
singletons have at most measure $1$. 
The \emph{first intensity function} of
the zero set $\mathsf{Z}= \{w \in \Pi^+:f_{\alpha}(w)=0\}$ of $f_{\alpha}$ is the function $\rho $ satisfying
\begin{equation*}
  \mathbb{E}[\card{\mathsf{Z} \cap U}]=\int_{U}\rho (z)\,\mathrm{d}z,
\end{equation*}
for every measurable subset $U\subseteq \Pi ^{+}$. 
The first intensity of
$\mathsf{Z}$ exists and can be computed from the
\emph{Edelman-Kostlan formula} \cite{sodin2000zeros,hough2009zeros, edelman1995many} as
\begin{equation}
\rho (z)=\frac{\alpha}{4\pi y^{2}}.
\label{eq:onepi}
\end{equation}
This means that the zeros  are distributed according to a multiple of
the hyperbolic area density on the upper half-plane. 
Besides this rough
description, the zeros of GAFs are known to be quite
rigid: the events where the concentration of the zeros deviates
significantly from what is prescribed by the first intensity are very
unlikely (see, e.g., the large deviations estimates in \cite{sodin2000zeros} and \cite{offord1967distribution}, the variance calculation in \cite{bu15}, and the hole probabilities in \cite{HyperbolicHole}).
Moreover, the interaction between zeros depends only on their \mbox{(pseudo-)}hyperbolic distance. 
In particular, the pair correlation function \cite[Sec.~4.3]{ilpest08} is given by \cite[eq.~(5.1.3)]{hough2009zeros}
\ba \label{eq:paircorr}
g(z,w) 
& = g(r)  \\
& = (1-s^{\alpha})^{-3} \Big( 1+(\alpha^2 - 2 \alpha -2)(s^{\alpha}+s^{2+2\alpha})
\notag \\*
& \quad 
+ (\alpha+1)^2 (s^{2\alpha}+s^{2+\alpha})
-2\alpha^2(s^{1+\alpha}+s^{1+2\alpha})+s^{2+3\alpha} \Big)
\notag
\ea
where $s=1-r^2$ and $r$ is the pseudo-hyperbolic distance between $z$ and $w$ given by
\be
  d_{\mathrm{ph}}(z,w)= \bigg\lvert \frac{z-w}{z-\bar{w}} \bigg \rvert\,.
\ee
Evaluating \eqref{eq:paircorr} numerically is surprisingly unstable as the positive and negative terms in the numerator turn out to be almost equal especially for values of $s$ close to one.
However, some basic algebraic manipulations show that the pair correlation function can be rewritten as 
\be \label{eq:paircorralt}
g(r)  
 = \frac{s^{\alpha} \big(\alpha (1-s)-s(1-s^{\alpha})\big)^2 + \big(\alpha s^{\alpha}(1-s)-(1-s^{\alpha})\big)^2}%
{(1-s^{\alpha})^{3}}\,.
\ee

Using the relation between the hyperbolic GAF and the wavelet transform given by~\eqref{eq:wavelet_wngaf}, the discretization of the wavelet transform of white noise presented in Section~\ref{sec:dwt}, and  Lemma~\ref{lem:locminzero} to identify the zeros in the discrete scalogram as local minima, enables us to simulate the zero pattern of a realization of a hyperbolic GAF. 
The snapshot of the zero set is presented in Figure~\ref{fig:zerohypgaf}.
\begin{figure}%
\includegraphics[width=1\textwidth,trim={1.5cm 7.3cm 1.5cm 7.1cm},clip]{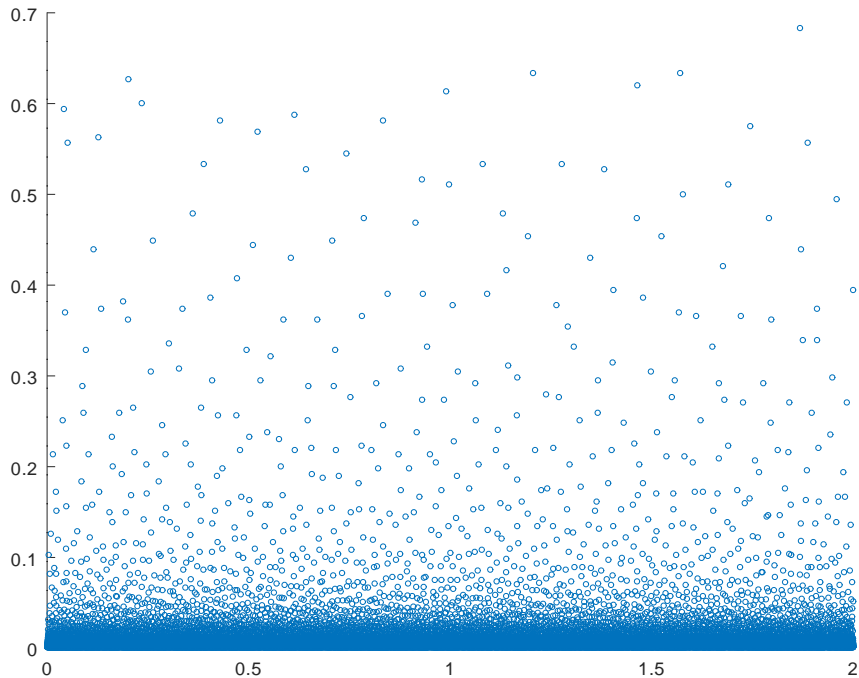}
\caption{Zero pattern of one realization of a hyperbolic GAF.}%\vspace{-6mm}}
\label{fig:zerohypgaf}
\end{figure}
That this simulation technique indeed results in a point pattern that follows the statistics predicted by \eqref{eq:onepi} and \eqref{eq:paircorr}, will be seen in our experiments in Section~\ref{sec:experiments}.

We propose two methods to find deviations from the zero pattern of white noise:
The first method is based on \eqref{eq:onepi} and tries to find areas where the local intensity of a given zero pattern deviates significantly from the first intensity function \eqref{eq:onepi} of pure noise. 
Here, we can quantify what ``significantly'' means based on the variance formula given in \cite{bu15}.

The second approach is based on the pair correlation function. 
We first propose a local estimator for the pair correlation function of a given zero pattern.
Furthermore, we derive in closed form the expectation of certain statistics of the point process that count point interactions in a certain region. 
In contrast to the pair correlation function these statistics are not approximate but exact and thus their deviation for a given  point pattern from the expected value is a more reliable method to identify signal components. 
In our experiments we will see that indeed the estimation of the pair correlation function is not reliable in particular for small sets whereas the proposed statistics are very stable for white noise signals.

%%%%%%%%%%%%%%%%%%%%%%%%%%%%%%%%%%%%%%%%%
\section{Estimation of the First Intensity Function}\label{sec:estfif}
%%%%%%%%%%%%%%%%%%%%%%%%%%%%%%%%%%%%%%%%%

For a given point $z\in \Pi^+$, we count the number of points in hyperbolic disks $D_\mathrm{ph}(z,r)$ of pseudo-hyperbolic radius $r$ around $z$. 
Due to the defining property of the first intensity function, the expected number of points in this disk for a pattern generated by white noise is given by
\be
  \mu_{r}
  = \mathbb{E}[\card{\mathsf{Z} \cap D_\mathrm{ph}(z,r)}]
  = \int_{D_\mathrm{ph}(z,r)}\frac{\alpha}{4\pi y^{2}}\,\mathrm{d}z
  = \frac{\alpha r^2}{1-r^2}
\ee
where we used that a disk of pseudo-hyperbolic radius $r$ has hyperbolic area $\frac{ r^2}{4\pi (1-r^2)}$ \cite[Sec.~2.3]{hough2009zeros}.
Furthermore, we can even calculate the variance of the random quantity $\card{\mathsf{Z} \cap D_\mathrm{ph}(z,r)}$ as \cite[Lem.~5]{bu15}
\ba
  \sigma^2_{r}
  & = \var[\card{\mathsf{Z} \cap D_\mathrm{ph}(z,r)}] \notag \\
  & =  \frac{\alpha^2 r^4}{2 \pi (1-r^2)^2}
        \int_{-\pi}^{\pi} \frac{2(1 - r^2)^{2\alpha}(1 - \cos t)}%
        {\big(\abs{1 -r^2 e^{it}}^{2\alpha}-((1-r^2)^{2 \alpha})\big) \abs{1-r^2 e^{it}}^2 }
        \,\mathrm{d}t \,.
\ea
Without any further knowledge the best probability bound on deviations from the mean is given by Chebyshev's inequality
\be
  \Pr\bigg[\frac{\abs{\card{\mathsf{Z} \cap D_\mathrm{ph}(z,r)}-\mu_{r}}}{\sigma_{r}}\geq \delta \bigg]
  \leq \frac{1}{\delta^2}\,.
\ee
Although we expect this bound to be quite loose in our scenario, it allows us to give rigorous statements on the probability that the zero pattern was not generated by white noise.
A simple procedure to do so is to decide on a finite number of centers $z_k$, radii $r_k$, and deviations $\delta_k$ before the experiment is performed and count the number of zeros in $D_\mathrm{ph}\big( z_k,r_k \big)$.
The probability that any of these counts deviates by more than $\sigma_{r_k} \delta_k$ from the means $\mu_{r_k}$ can be trivially upper-bounded by
\be
  \Pr\bigg[\bigcup_{k}\bigg\{\frac{\abs{\card{\mathsf{Z} \cap D_\mathrm{ph}(z_k,r_k)}-\mu_{r_k}}}{\sigma_{r_k}}\geq \delta_k \bigg\} \bigg]
  \leq \sum_k \frac{1}{\delta_k^2}\,.
\ee
Note that we either have to use quite large $\delta$ or only few $k$ to obtain a useful bound. 

Similarly, we can use the deviation of local estimates of the first intensity function from the expected mean $\mu_r$ normalized by the standard deviation $\sigma_r$ to create a mask that allows us to identify areas where we expect signal components.
Combining the masks for various radii $r>0$ enables a smoother filtering procedure.
More specifically, for a given zero pattern $Z$, we calculate for $w\in \Pi^+$
\be\label{eq:maskint}
  q(w) = \min\bigg\{
  \max \bigg\{
  a
  \bigg(\sum_{k=1}^K\frac{ \big( \card{Z \cap D_\mathrm{ph}(w,r_k)}-\mu_{r_k}\big)^2}{K\sigma_{r_k}^2}
  \bigg) 
  -b
  , 0 \bigg\}
  , 1
  \bigg\}
\ee
defining a filtering mask for every point on $\Pi^+$.
Here, the factor $a>0$ and threshold $b\geq 0$ are design parameters that can be chosen based on the application. 
 In our experiments in Section~\ref{sec:experiments}, we choose $a=1$ and $b=4.65$. 
This value of $b$ was obtained in simulations as the $0.999$ quantile of the calculated statistic for white noise and thus results in almost complete elimination of noise. Only the most prominent signal components are preserved.
For softer filtering the threshold can also be set to lower values of $b$ or even  $b=0$.

%%%%%%%%%%%%%%%%%%%%%%%%%%%%%%%%%%%%%%%%%
\section{Estimation of the Pair Correlation Function}\label{sec:estpcf}
%%%%%%%%%%%%%%%%%%%%%%%%%%%%%%%%%%%%%%%%%

The second moment properties of a point pattern are most conveniently described by the pair correlation function \cite[Sec.~4.3]{ilpest08}.
For the zero set of white noise, there is a closed form expression of the pair correlation function available and thus we can try to locally estimate the pair correlation function of a given point pattern to identify regions of deviation from white noise.
There is abundant literature on the estimation of the pair correlation function for a given point pattern, but often they are designed under the assumption of a stationary point process generating the pattern.
Furthermore, the methods we are aware of assume that a point pattern is observed in some observation window $W$ and use all points observed in $W$ to estimate the pair correlation function. 
In particular, these methods assume that one does not have any knowledge of the point pattern outside of $W$ and usually have to deal with boundary effects.
We, however, want to estimate the pair correlation function locally and thus restrict ourselves to a small window $W$ but still have access to the points outside of $W$.
Thus, we do not have to deal with boundary effects.
This considerably simplifies the estimation procedure.
%and the estimation becomes surprisingly simple.

We thus deviate from classical estimators in that we have to deal with the hyperbolic background measure, i.e., the point process is not stationary, and that we want to estimate the pair correlation function in a given window $W$ but also have access to points outside of $W$.
As suggested in \cite{ilpest08}, we will estimate the second factorial moment density $\rho^{(2)}(z,w)$ and the first intensity $\rho(z)$ to obtain an estimate of the pair correlation function $g(z,w)= \rho^{(2)}(z,w)/(\rho(z)\rho(w))$.

\subsection{Second Factorial Moment Density}
The definition of the second factorial moment density $\rho^{(2)}$ of a point process $\mathsf{Z}$ guarantees that
\be
  \E\bigg[\sum_{(z\neq w)\in \mathsf{Z}} f(z,w)\bigg] = \int\int f(z,w) \rho^{(2)}(z,w)\,\mathrm{d}z\mathrm{d}w
\ee
for any measurable $f\colon \C^2 \to \R^+$.
%For our estimation of the second factorial moment density, we first choose a localization set $W\subseteq \Pi^+$, e.g., a hyperbolic disk.
We assume that the pair correlation function depends only on the hyperbolic distance between $z$ and $w$, i.e., $g(z,w)=g(d_{\mathrm{ph}}(z,w))$, whereas the first intensity function is constant  with respect to hyperbolic area, i.e., $\rho(z) = \tilde{\rho}/y^2$. 
Thus, we can rewrite $\rho^{(2)}(z,w) = \frac{\tilde{\rho}^{(2)}(d_{\mathrm{ph}}(z,w))}{y^2 v^2}$.
To obtain an estimate of $\tilde{\rho}^{(2)}(r_0)$ for $r_0\in (0,1)$, we choose 
$
f(z,w) = \ind_{W}(w) \ind_{D_\mathrm{ph}(w,r_0+h)\setminus D_\mathrm{ph}(w,r_0-h)}(z)$ for a width $h\in (0,r_0]
$.
Thus, we obtain
\be \label{eq:calcpcf}
  \E\bigg[\sum_{\substack{w\in \mathsf{Z}\cap W \\ z\in \mathsf{Z}:\abs{d_{\mathrm{ph}}(z,w)-r_0}<h}} 1\bigg] 
  = \int_{W}\int_{\{z:\abs{d_{\mathrm{ph}}(z,w)-r_0}<h\}} \rho^{(2)}(z,w)\,\mathrm{d}z\mathrm{d}w.
\ee
We illustrate the set $\{z:\abs{d_{\mathrm{ph}}(z,w)-r_0}<h\}$ for a given $w$ in Fig.~\ref{fig:setw}.
\begin{figure}%
\begin{tikzpicture}
	\filldraw[fill=blue!40!white, draw=black] (0,0) circle (1.5cm);
  \filldraw[fill=white, draw=black] (0,0) circle (1cm);
  \draw[black] (0,0) circle (1.25cm);
  \draw[thick,<->] (1,0) -- (1.5,0);
  \draw (1.25,0) node[anchor=south] {$2h$};
  \draw[thick,->] (0,0) -- (0,1.25);
  \draw (0,0.65) node[anchor=east] {$r_0$};
  \filldraw[black] (0,0) circle (0.06cm) node[anchor=north] {$w$};
\end{tikzpicture}
\caption{Illustration of the set $\{z:\abs{d_{\mathrm{ph}}(z,w)-r_0}<h\}$. For simplicity we illustrate the set for the Euclidean distance in place of the pseudo-hyperbolic distance.
 For a given $w$, the function $f(z,w)$ is equal to one if and only if the point $z$ belongs to the blue-shaded ring.}%
\label{fig:setw}%
\end{figure}
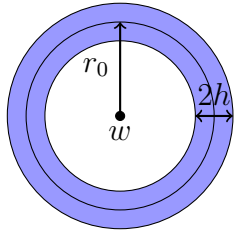
%Since the pair correlation function is defined as $g(z,w) = \frac{\rho^{(2)}(z,w)}{\rho(z)\rho(w)}$ and $\rho$ is given by \eqref{eq:onepi}, this can be rewritten as
%\be\label{eq:calcpcf}
%\E\bigg[\sum_{(z\neq w)\in N} f(z,w)\bigg] = \int\int f(z,w) g(z,w)  \frac{\alpha}{4\pi y^2} \,\mathrm{d}z \frac{\alpha}{4\pi v^2}\,\mathrm{d}w
%\ee
%where $z = x+iy$ and $w=u+iv$.
We change the coordinates $(x,y)$ of $z$ to a kind of pseudo-hyperbolic polar coordinates $(r,\phi)$ centered at $w$ by the change of variables 
\be
\begin{pmatrix}
	x \\
  y
\end{pmatrix}
= 
\begin{pmatrix}
	u+\frac{2rv}{1-r^2} \sin \phi \\
  \frac{1+r^2}{1-r^2}v + \frac{2rv}{1-r^2} \cos \phi
\end{pmatrix}\,.
\ee
The corresponding Jacobian determinant is given by $\frac{4rv^2(1+r^2+2r\cos \phi)}{(1-r^2)^3}$.
Using this transformation in \eqref{eq:calcpcf} results in 
\ba
& \E\bigg[\sum_{\substack{w\in \mathsf{Z}\cap W \\ z\in \mathsf{Z}:d_{\mathrm{ph}}(z,w)\in (r_0-h, r_0+h)}} 1\bigg] 
\notag \\*
& = \int_W \int_{r_0-h}^{r_0+h} \int_0^{2\pi} \tilde{\rho}^{(2)}(r) \frac{4rv^2(1+r^2+2r\cos \phi)}{(1-r^2)^3 (\frac{1+r^2}{1-r^2}v + \frac{2rv}{1-r^2} \cos \phi)^2} \mathrm{d}\phi \mathrm{d}r \frac{1}{ v^2}\,\mathrm{d}w
\notag \\
& = 4 \int_{r_0-h}^{r_0+h} \frac{r\,\tilde{\rho}^{(2)}(r) }{1-r^2}\int_0^{2\pi} \frac{1}{1+r^2+2r\cos \phi}   \,\mathrm{d}\phi \mathrm{d}r \int_W \frac{1}{v^2}\,\mathrm{d}w
\notag \\
& = 4 \int_{r_0-h}^{r_0+h} \frac{r\,\tilde{\rho}^{(2)}(r) }{1-r^2}\frac{2\pi }{1-r^2} \mathrm{d}r \int_W \frac{1}{v^2}\,\mathrm{d}w
\notag \\
& = 8\pi \int_{r_0-h}^{r_0+h} \frac{r\,\tilde{\rho}^{(2)}(r)}{(1-r^2)^2} \mathrm{d}r \int_W \frac{1}{v^2}\,\mathrm{d}w
\notag \\
& \approx 16 \pi\frac{hr_0\,\tilde{\rho}^{(2)}(r_0) }{(1-r_0^2)^2} \int_W \frac{1}{v^2}\,\mathrm{d}w\,.
\label{eq:estmyf}
\ea
Choosing $W$ as a disk $D_\mathrm{ph}(w_1,r_1)$ around an estimation center $w_1$ and  of pseudo-hyperbolic radius $r_1\in (0,1)$, the integral in this approximation is  $\frac{4\pi r_1^2}{1-r_1^2}$ and we obtain 
\be
 \E\bigg[\sum_{ w\in \mathsf{Z}\cap D_\mathrm{ph}(w_1,r_1)  }  \card{\{z\in \mathsf{Z}:\abs{d_{\mathrm{ph}}(z,w)-r_0}<h\}}\bigg] 
 \approx   \frac{4^3 \pi^2 h  r_1^2 r_0\,\tilde{\rho}^{(2)}(r_0) }{(1-r_1^2)(1-r_0^2)^2} \,.
\notag
\ee
Thus, for a given zero pattern $Z$ we have the following estimator of $\rho^{(2)}$
\ba
  & \hat{\rho}_{r_0,r_1}^{(2)}(z_1,w_1)
  \label{eq:estrho2} \\
  & = \frac{(1-r_1^2)(1-r_0^2)^2}{4^3 \pi^2 h  r_1^2 r_0 v_1^2 y_1^2} \sum_{ w\in Z\cap D_\mathrm{ph}(w_1,r_1)  }  \card{\{z\in Z:\abs{d_{\mathrm{ph}}(z,w)-r_0}<h\}}.
  \notag 
\ea

\subsection{First Intensity Function}
To estimate the pair correlation function, we also need an estimator of the first intensity function at $z_1$ and $w_1$.
Similarly to $\tilde{\rho}^{(2)}$, we assume that the intensity is proportional to the hyperbolic background measure, i.e., we set $\rho(z) = \frac{\tilde{\rho}}{y^2}$ where $\tilde{\rho}$
is constant.
Based on the defining property of the first intensity, we have
\be
  \E[\card{\mathsf{Z} \cap D_\mathrm{ph}(w_1,r_1)}]=\int_{D_\mathrm{ph}(w_1,r_1)}\frac{\tilde{\rho}}{y^2}\,\mathrm{d}z = \frac{\tilde{\rho} 4 \pi r_1^2}{1-r_1^2},
\ee
and, in turn, the estimator
\be\label{eq:estrho}
  \hat{\rho}_{r_1}(w_1) = \card{Z \cap D_\mathrm{ph}(w_1,r_1)} \frac{1-r_1^2}{ 4 \pi r_1^2 v_1^2}
\ee
for $\rho$.
To finally obtain an estimator that does only depend on the localization set $D_\mathrm{ph}(w_1,r_1)$ and the pseudo-hyperbolic radius $r_0$, we do not estimate $\rho(z_1)$ by \eqref{eq:estrho} but use the white noise first intensity 
$\rho (z_1)=\frac{\alpha}{4\pi y_1^{2}}$ from \eqref{eq:onepi} instead.

\subsection{Pair Correlation Function}
Combining \eqref{eq:estrho2}, \eqref{eq:estrho}, and \eqref{eq:onepi}, we obtain the following local estimator of the  pair correlation function
\ba
  \hat{g}_{w_1,r_1}(r_0) 
  & = \frac{\hat{\rho}_{r_0,r_1,h}^{(2)}(z_1,w_1)}{\rho(z_1)\hat{\rho}_{r_1}(w_1)} \notag \\
  & = \frac{  (1-r_0^2)^2}{4 \alpha h   r_0  } \frac{\sum_{ w\in Z\cap D_\mathrm{ph}(w_1,r_1)  }  \card{\{z\in Z:\abs{d_{\mathrm{ph}}(z,w)-r_0}<h\}}}{\card{Z \cap D_\mathrm{ph}(w_1,r_1)}}
  \label{eq:paircorrest} 
\ea 
which does no longer depend on the second point $z_1$ but is only a function of the localization set $W = D_\mathrm{ph}(w_1,r_1)$ and the radius $r_0$.
Note again that this estimator does not only depend on the zeros in $W$ but potentially on all zeros in the larger circle $D_\mathrm{ph}(w_1,r_1+r_0)$.

Since we know the pair correlation function for white noise, we can use a filtering procedure similar to Section~\ref{sec:estfif} to identify time-scale components that most likely do not result from  white noise.
Similarly to \eqref{eq:maskint}, we define a mask by 
%\be\label{eq:maskpc}
  %q^{(2)}(w_1) = \min\bigg\{\sum_{k_0=1}^{K_0}\sum_{k_1=1}^{K_1}\frac{ \abs{g(r_{0,k_0}) - \hat{g}_{w_1,r_{1,k_1}}(r_{0,k_0})}}{K_0 K_1}, 1
  %\bigg\}
%\ee
\ba\label{eq:maskpc}
  & q^{(2)}(w_1) \notag \\*
  & = \min\bigg\{
  \max \bigg\{
  \bigg(\sum_{k=1}^{K_0}\sum_{\ell=1}^{K_1}\frac{a_{k,\ell} \big(g(r_{0,k}) - \hat{g}_{w_1,r_{1,\ell},h}(r_{0,k})\big)^2}{K_0 K_1}
  \bigg) 
  -b
  , 0 \bigg\}
  , 1
  \bigg\}
\ea
where $g$ was defined in \eqref{eq:paircorr} and the local estimator $\hat{g}$ in \eqref{eq:paircorrest}.
The factor $a_{k,\ell} $ here depends on the indices $k,\ell$ to enable a correction for the variance of the estimated pair correlation. 
In our experiments in Section~\ref{sec:experiments}, we estimate this variance based on simulations using white noise.

%%%%%%%%%%%%%%%%%%%%%%%%%%%%%%%%%%%%%%%%%
\subsection{Exact Statistics of Point Interactions}\label{sec:statpa}
%%%%%%%%%%%%%%%%%%%%%%%%%%%%%%%%%%%%%%%%%

Recall that in our approximation of the estimator of the second factorial 
moment measure above, we used in \eqref{eq:estmyf} an approximation that is based on the assumption that $\frac{r\,\tilde{\rho}^{(2)}(r)}{(1-r^2)^2}$ is almost affine for small variations of $r$.
However, since we know the pair correlation function and the first intensity function for a zero pattern generated by white noise, we can alternatively calculate the expected value in \eqref{eq:estmyf} exactly.
More specifically, we have
\ba
\E\bigg[\sum_{\substack{w\in \mathsf{Z}\cap D_\mathrm{ph}(w_1,r_1)  \\ z\in \mathsf{Z}:d_{\mathrm{ph}}(z,w)\in (a, b)}} 1\bigg] 
& = 8\pi \int_{a}^{b} \frac{r\,\tilde{\rho}^{(2)}(r)}{(1-r^2)^2} \mathrm{d}r \int_{ D_\mathrm{ph}(w_1,r_1)} \frac{1}{v^2}\,\mathrm{d}w
\notag \\
& = \frac{2 r_1^2\alpha^2}{1-r_1^2} \int_{a}^{b} \frac{r\,g(r) }{(1-r^2)^2} \mathrm{d}r  \,.
\notag 
\ea
Here, the integral can be calculated by substituting $1-r^2=s$ and inserting \eqref{eq:paircorralt}
\ba
& \frac{2 r_1^2\alpha^2}{1-r_1^2} \int_{a}^{b} \frac{r\,g(r) }{(1-r^2)^2} \mathrm{d}r
%& = \frac{\alpha^2}{4\pi} \int^{1-a^2}_{1-b^2} \frac{ s^{\alpha} \big(\alpha (1-s)-s(1-s^{\alpha})\big)^2 + \big(\alpha s^{\alpha}(1-s)-(1-s^{\alpha})\big)^2
% }{s^2(1-s^{\alpha})^{3}} \mathrm{d}s \int_W \frac{1}{v^2}\,\mathrm{d}w
\notag \\*
& =  \frac{r_1^2\alpha^2}{1-r_1^2} \bigg[ \frac{ (\alpha+1) s^{\alpha} (1-s)^2-(1-s^{\alpha+1})^2
 }{s(1-s^{\alpha})^{2}} \bigg]^{1-a^2}_{s= 1-b^2}  \,.
\notag
\ea
For $s  \nearrow 1$, the expression in brackets converges to ${-}\frac{\alpha+ 1}{\alpha}$ and for other values of $s\in (0,1)$ it can simply be evaluated.
Proceeding as above, we see that the estimator in \eqref{eq:paircorrest} should actually not be compared to the pair correlation function $g$ but the modified term
\ba
   \tilde{g}(r_0, h) 
  %\notag \\* 
  & =  \frac{  (1-r_0^2)^2}{4  h   r_0  }  
  \Big( 
  \tfrac{ (\alpha+1) (1- (r_0-h)^2)^{\alpha} (r_0-h)^4-(1-(1- (r_0-h)^2)^{\alpha+1})^2
 }{(1- (r_0-h)^2)(1-(1- (r_0-h)^2)^{\alpha})^{2}} 
\notag \\
& \quad -
  \tfrac{ (\alpha+1) (1- (r_0+h)^2)^{\alpha} (r_0+h)^4-(1-(1- (r_0+h)^2)^{\alpha+1})^2
 }{(1- (r_0+h)^2)(1-(1- (r_0+h)^2)^{\alpha})^{2}}
  \Big)
  \notag
\ea 
provided $r_0-h>0$, and 
\ba
%& \E\bigg[\sum_{ w\in Z\cap D_\mathrm{ph}(w_1,r_1)  }  \card{\{z\in Z:\abs{d_{\mathrm{ph}}(z,w) }<r_0\}}\bigg] 
%\notag \\*
\tilde{g}(r_0, h)
& = \frac{  (1-r_0^2)^2}{4  h   r_0  }
\Big(
{-}\tfrac{\alpha+ 1}{\alpha} 
-
  \tfrac{ (\alpha+1) (1- (r_0+h)^2)^{\alpha} (r_0+h)^4-(1-(1- (r_0+h)^2)^{\alpha+1})^2
 }{(1- (r_0+h)^2)(1-(1- (r_0+h)^2)^{\alpha})^{2}}
  \Big)  
\notag
\ea
for the special case $r_0-h=0$.
As we will see in the simulations below, these $\tilde{g}$ are significantly closer to the estimates $\hat{g}$ for simulated white noise than the true pair correlation function $g$ in particular for small values of $r_0$.
Thus, $\tilde{g}$ should be used in the filtering procedure described in \eqref{eq:maskpc} in place of $g$.

%%%%%%%%%%%%%%%%%%%%%%%%%%%%%%%%%%%%%%%%%
\section{Experiments}\label{sec:experiments}
%%%%%%%%%%%%%%%%%%%%%%%%%%%%%%%%%%%%%%%%%

We first illustrate the accuracy of the proposed estimators in the discrete white noise setting.%
\footnote{
The scripts we used for conducting our experiments is provided on \url{https://github.com/gkoliander/WaveletPPP}.}
To this end, we generate $L = 88200$ independent standard normal samples that we interpret as samples of a $2$s long signal.
We apply the discrete continuous WT with mother wavelet $\mwlet_{\alpha}$ and $\alpha = 300$ and interpret the unique peak of $\mwlet_{\alpha}$ in the frequency domain as the frequency associated with a given scale (cf.~\cite[Sec.~IV]{hokopr19}).
To identify set of zeros $Z$ in the discretized scalogram, we use the fact that zeros and local minima coincide (see Lemma~\ref{lem:locminzero}) and find time-scale points where the modulus of the discrete continuous WT is less than at all 4 neighboring points (using 8 neighbors did not improve accuracy and resulted only in higher complexity).

We count the number of points $\card{Z \cap D_\mathrm{ph}(w_1,r_1)}$ for all discrete time-scale points $w_1$ and various radii $r_1$, and calculate $\hat{g}_{w_1,r_1}(r_0)$ in \eqref{eq:paircorrest} for various radii $r_0$ and $r_1$ and all discrete time-scale points where the denominator in \eqref{eq:paircorrest} is nonzero.
To discard the most significant boundary effects, we further restrict to time-scale points $w_1$ that have a minimum pseudo hyperbolic distance of the maximal value of $r_0+r_1$ to the boundary of our observation window.
The sample mean $\hat{\mu}_{r_1}$ and sample variance $\hat{\sigma}_{r_1}^2$ of $\card{Z \cap D_\mathrm{ph}(w_1,r_1)}$ for the remaining discrete time-scale points $w_1$ are given in Table~\ref{tab:firstint} in comparison to their proposed expectations.
 \begin{table}[tb]
\caption{Sample mean $\hat{\mu}_{r_1}$ and sample variance $\hat{\sigma}_{r_1}^2$ of the number of zeros in a disk of various radii $r_1$ in comparison to the exact mean and variance}\label{tab:firstint}
 \renewcommand{\arraystretch}{1.2}
 \centering
\begin{tabular}{|c|cccc|}
  \hline 
     $r_1$ & $\mu_{r_1}$ & $\hat{\mu}_{r_1}$ & $\sigma_{r_1}^2$ & $\hat{\sigma}_{r_1}^2$ \\
   \hline 
     $0.0768 $ & $1.781$ 		& $1.788$ 	& $0.531$ 		& $0.543$ \\
   \hline                                                                       	                	
     $0.1024 $ & $3.181$ 		& $3.183$ 	& $0.684$ 		& $0.697$ \\
   \hline 
     $0.1280 $ & $5$ 				& $4.985$ 	& $0.849$ 		& $0.859$ \\         
   \hline
     $0.1536 $ & $7.253$ 	& $7.244$ & $1.019$ 		& $1.033$ \\
   \hline
     $0.1793 $ & $9.959$ 	& $9.957$ & $ 1.194$ 		& $1.206$  \\
   \hline 
  \end{tabular}
 \renewcommand{\arraystretch}{1}
\end{table}
Similarly, we present the averaged values of $\hat{g}_{w_1,r_1,h}(r_0)$ in comparison to the true pair correlation function $g(r_0)$ and the corrected function  $\tilde{g}(r_0, h)$ in Table~\ref{tab:paircorr}.
 \begin{table}[thb]
   \caption{Sample means of $\hat{g}_{w_1,r_1,h}(r_0)$  for one realization of white noise over all valid $w_1$ for fixed $h=0.0427$ and various radii $r_0$ and $r_1$ in comparison to  the true pair correlation function $g(r_0)$ and the corrected function  $\tilde{g}(r_0, h)$.}\label{tab:paircorr}
 \renewcommand{\arraystretch}{1.2}
 \begin{center}
 \resizebox{\textwidth}{!}{% 
 \begin{tabular}{|c|cc|cccccc|}
  \hline 
    $r_0$ & $g(r_0)$ & $\tilde{g}(r_0, h)$ & \multicolumn{6}{c|}{Mean of $\hat{g}_{w_1,r_1,h}(r_0)$}  \\
   \hline 
     & & &   $r_1=$ & $0.0768 $ & $0.1024 $ &$0.1280 $ & $0.1536 $ & $0.1793 $ \\
   \hline
     $0.0427$ 		& $0.270$ 	& $0.489$ 	 & 	& $0.542$ & $0.504$ & $0.493$ & $0.492$ & $0.492$ \\
   \hline                                                 	                	
     $0.0854$ 		& $0.863$ 	& $0.861$ 	 & 	& $0.854$ & $0.854$ & $0.857$ & $0.858$ & $0.858$ \\
   \hline 
     $0.1280$ 		& $1.050$ 	& $1.022$ 	 & 	& $1.039$ & $1.026$ & $1.024$ & $1.024$ & $1.023$ \\         
   \hline 
  \end{tabular}
 \renewcommand{\arraystretch}{1}
}
\end{center}
\end{table}
We see that the problem described in Section~\ref{sec:statpa} is particularly prevalent for small radii and thus the correction will be used in the proceeding filtering schemes.
The sample variances of the estimator $\hat{g}_{w_1,r_1,h}(r_0)$ for various values of $r_0$ and $r_1$ are presented in Table~\ref{tab:paircorrvar}.
 \begin{table}[thb]
 \caption{Sample standard deviations of $\hat{g}_{w_1,r_1,h}(r_0)$ for one realization of white noise over all valid $w_1$ for fixed $h=0.0427$ and various radii $r_0$ and $r_1$.}\label{tab:paircorrvar}
\renewcommand{\arraystretch}{1.2}
 \begin{center}
 \begin{tabular}{|c|ccccc|}
  \hline 
    $r_0$ & $\hat{\sigma}_{r_{0},0.0768}$ & $\hat{\sigma}_{r_{0},0.1024}$ & $\hat{\sigma}_{r_{0},0.1280}$ & $\hat{\sigma}_{r_{0},0.1536}$ & $\hat{\sigma}_{r_{0},0.1793}$  \\
   \hline
     $0.0427$ 		& $0.208$ & $0.202$ & $0.176$ & $0.149$ & $0.127$ \\
   \hline                                                 	                	
     $0.0854$ 		& $0.188$ & $0.148$ & $0.121$ & $0.101$ & $0.085$\\
   \hline 
     $0.1280$ 		& $0.165$ & $0.121$ & $0.097$ & $0.083$ & $0.073$ \\         
   \hline 
  \end{tabular}
 \renewcommand{\arraystretch}{1}
\end{center}
\end{table}
These are used in the filtering procedure below to normalize the deviation from the true pair correlation function $g(r_0)$ and the corrected function  $\tilde{g}(r_0, h)$ dependent on the radii $r_0$ and $r_1$.

We perform some experiments of the proposed tests and filtering approaches to illustrate their strength and weaknesses.
The first example we use to illustrate our approach is the male English voice signal~50 in~\cite{sqam} from $0.4$s to $2.4$s. 
In Fig.~\ref{fig:voice}, we show the scalogram and masks based on first intensity estimation and the pair correlation estimation for $\alpha=300$.
\begin{figure}[tbp]%
\begin{minipage}[c]{0.45\textwidth}
\includegraphics[width=1\textwidth,trim={1.5cm 7.8cm 2.3cm 7.1cm},clip]{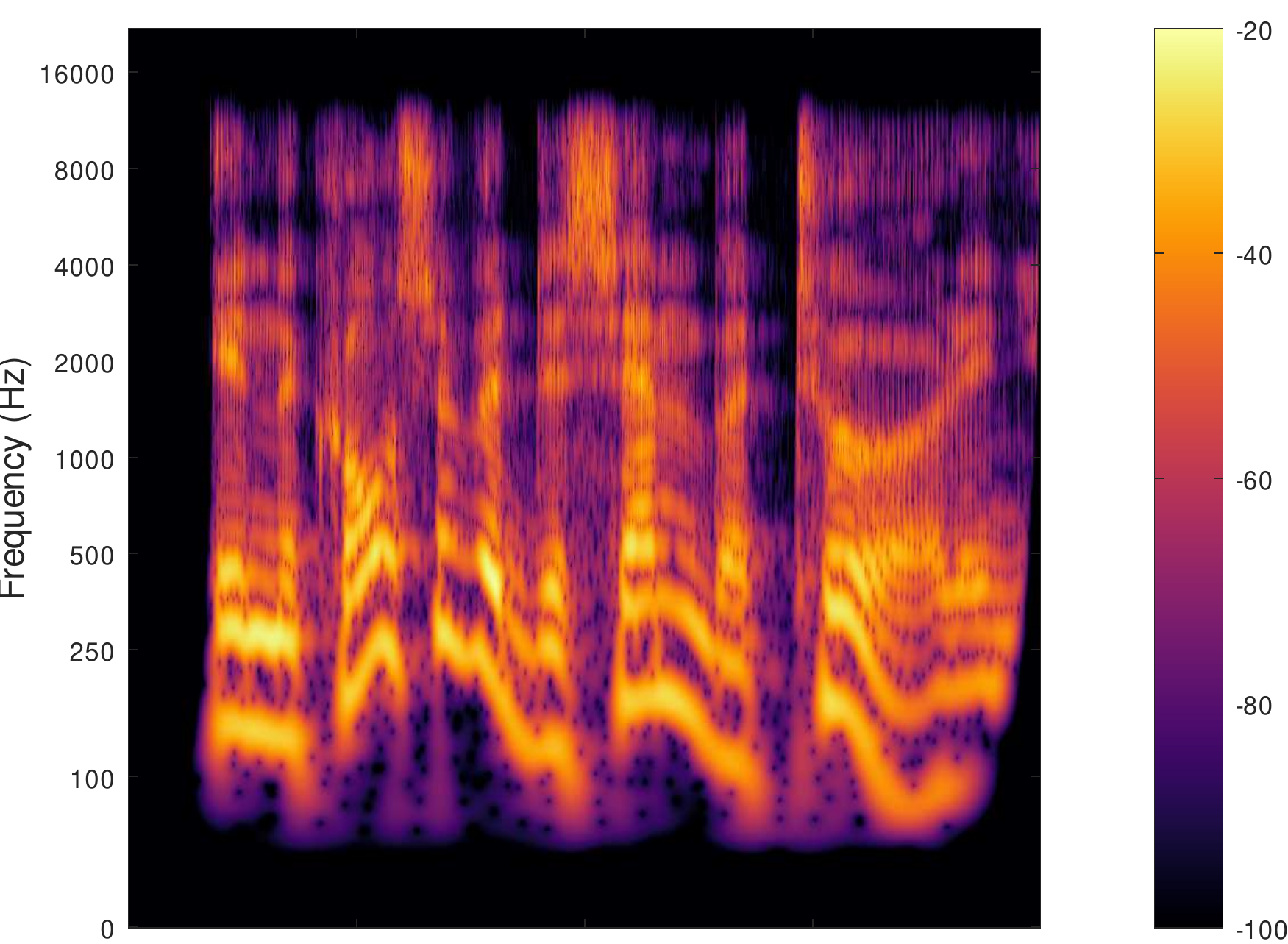}
\end{minipage}
\begin{minipage}[c]{0.45\textwidth}
\includegraphics[width=1\textwidth,trim={1.5cm 7.8cm 2.3cm 7.1cm},clip]{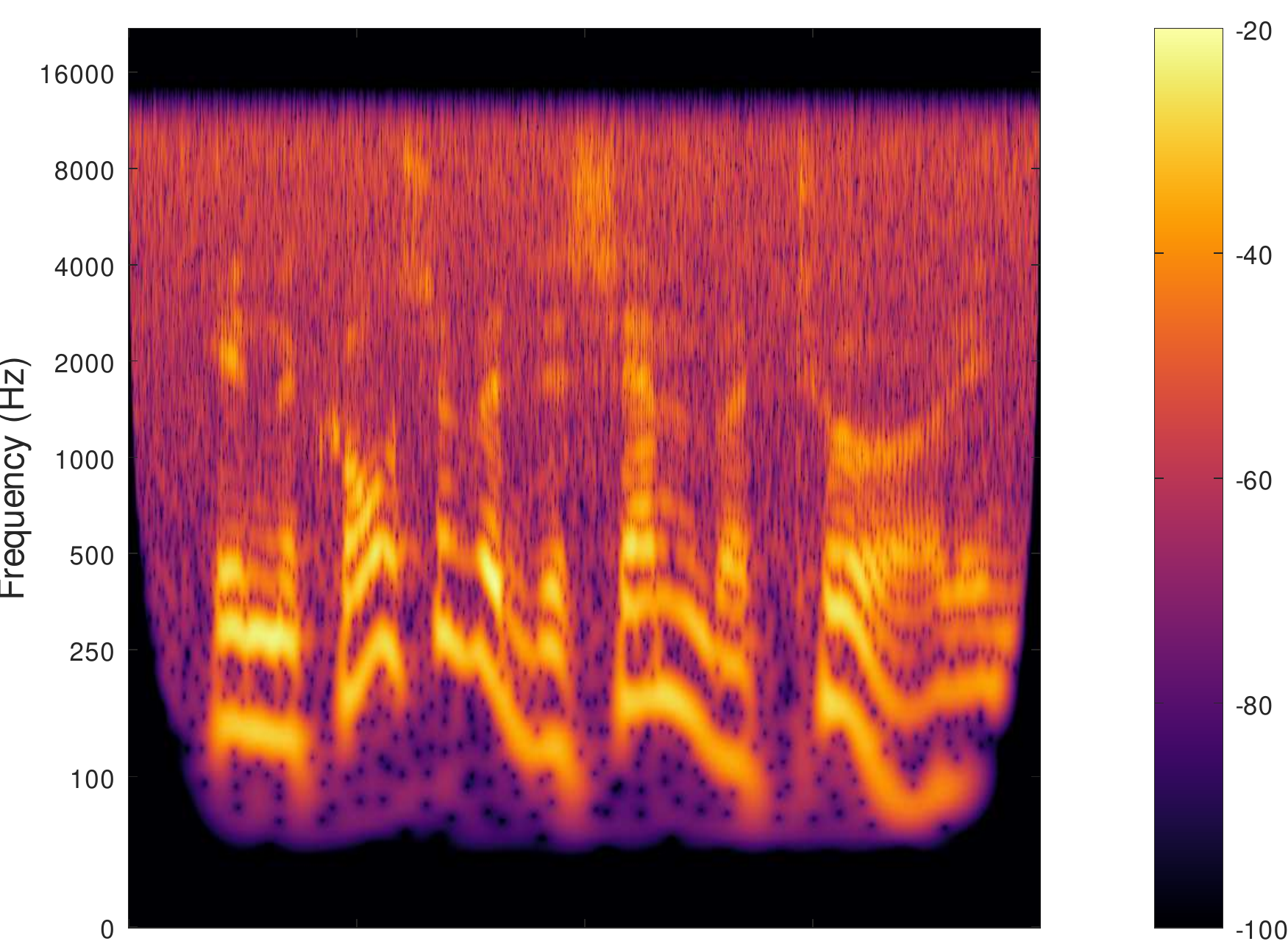}
\end{minipage}
\begin{minipage}[c]{0.45\textwidth}
\includegraphics[width=1\textwidth,trim={1.5cm 7.8cm 2.3cm 7.1cm},clip]{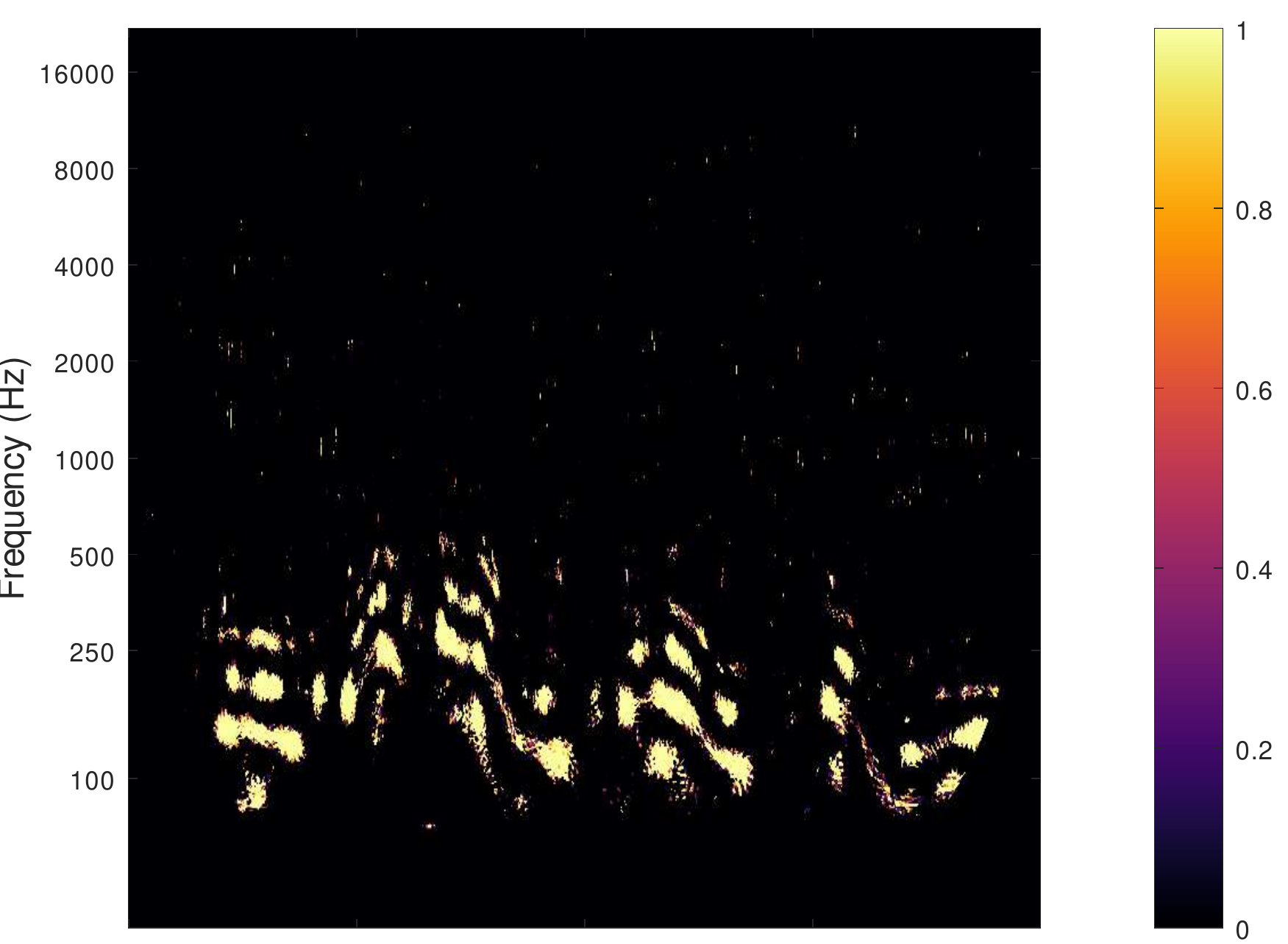}
\end{minipage}
\begin{minipage}[c]{0.45\textwidth}
\includegraphics[width=1\textwidth,trim={1.5cm 7.8cm 2.3cm 7.1cm},clip]{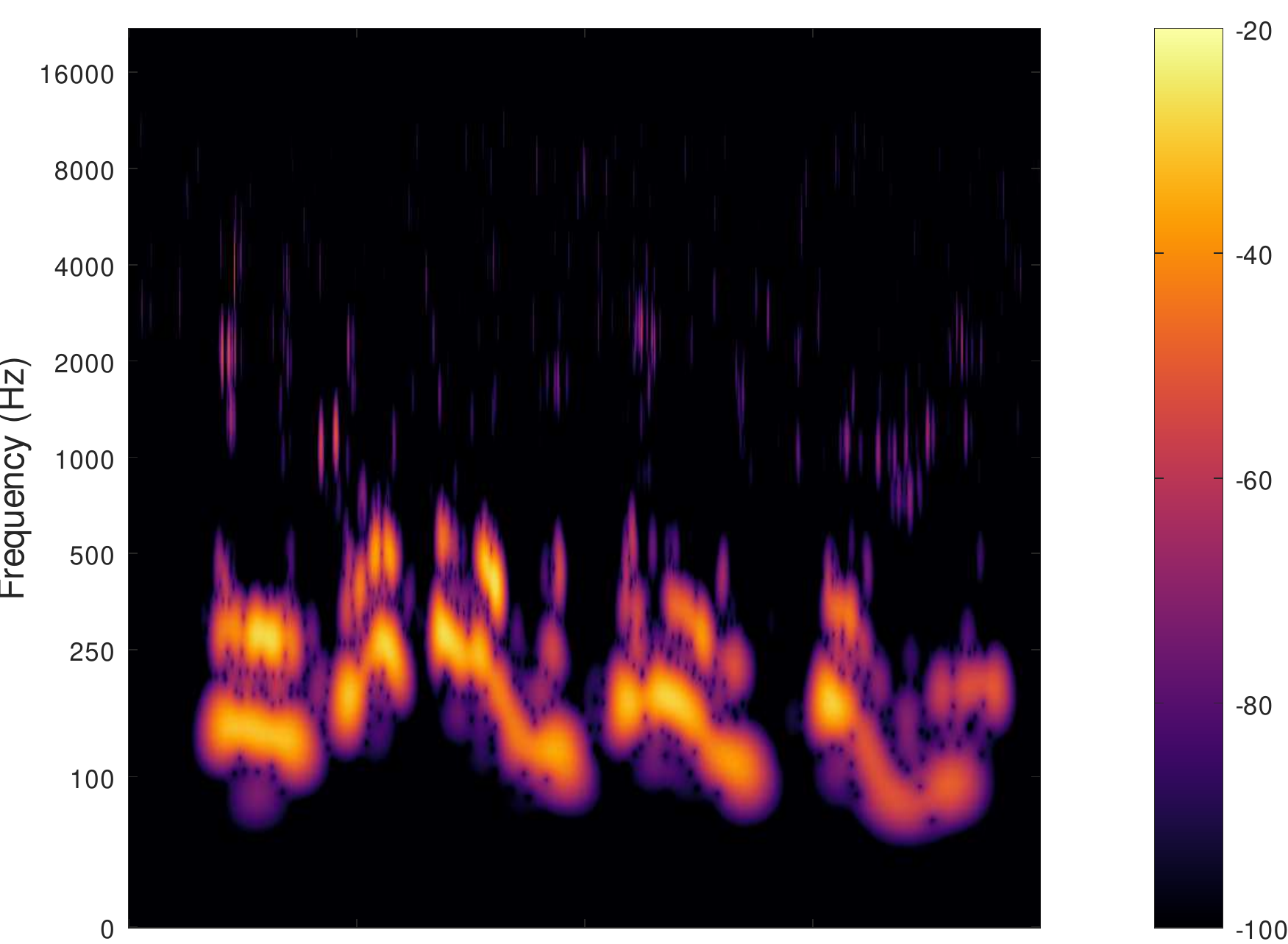}
\end{minipage}
\begin{minipage}[c]{0.45\textwidth}
\includegraphics[width=1\textwidth,trim={1.5cm 7.8cm 2.3cm 7.1cm},clip]{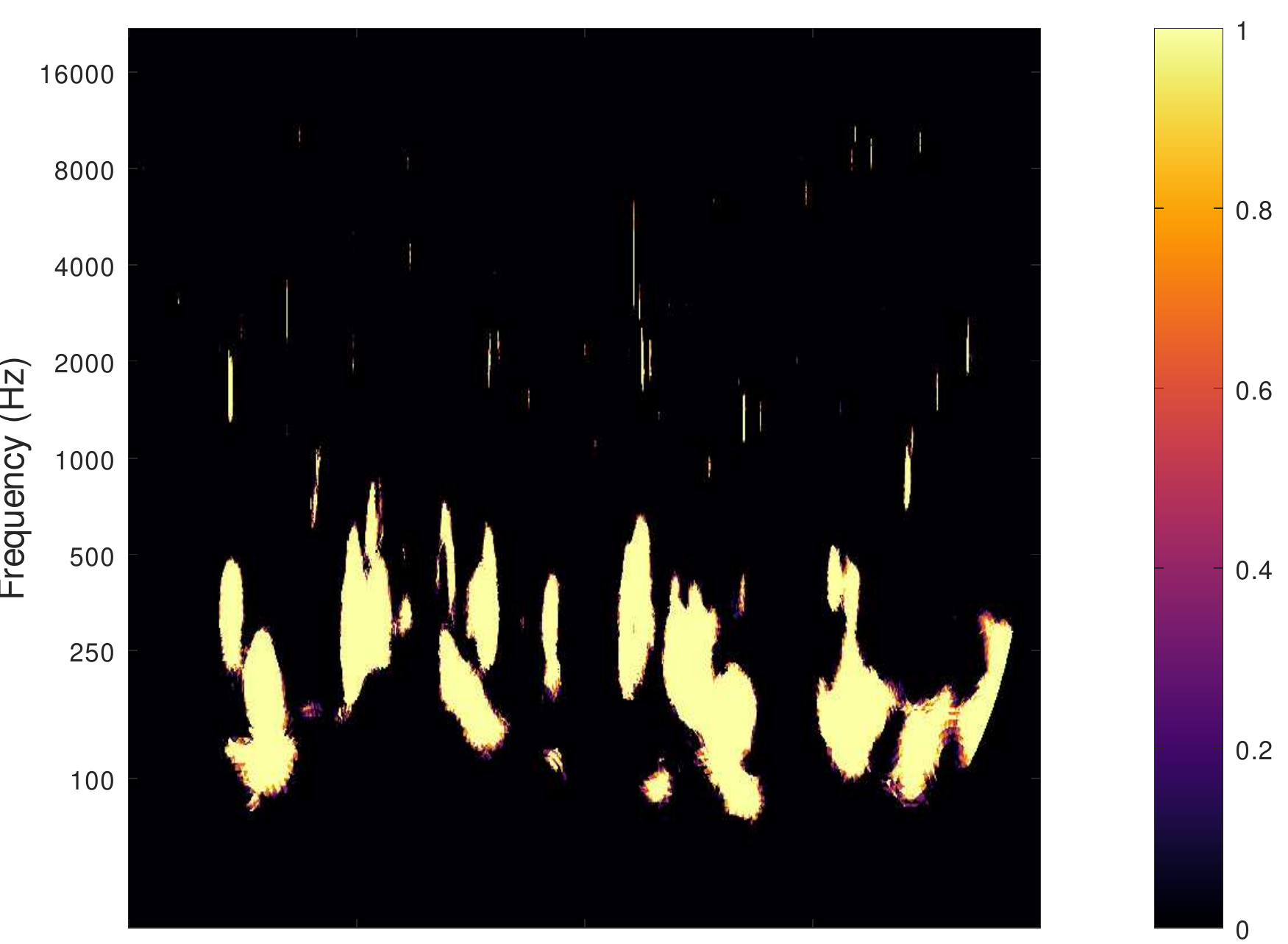}
\end{minipage}
\begin{minipage}[c]{0.45\textwidth}
\includegraphics[width=1\textwidth,trim={1.5cm 7.8cm 2.3cm 7.1cm},clip]{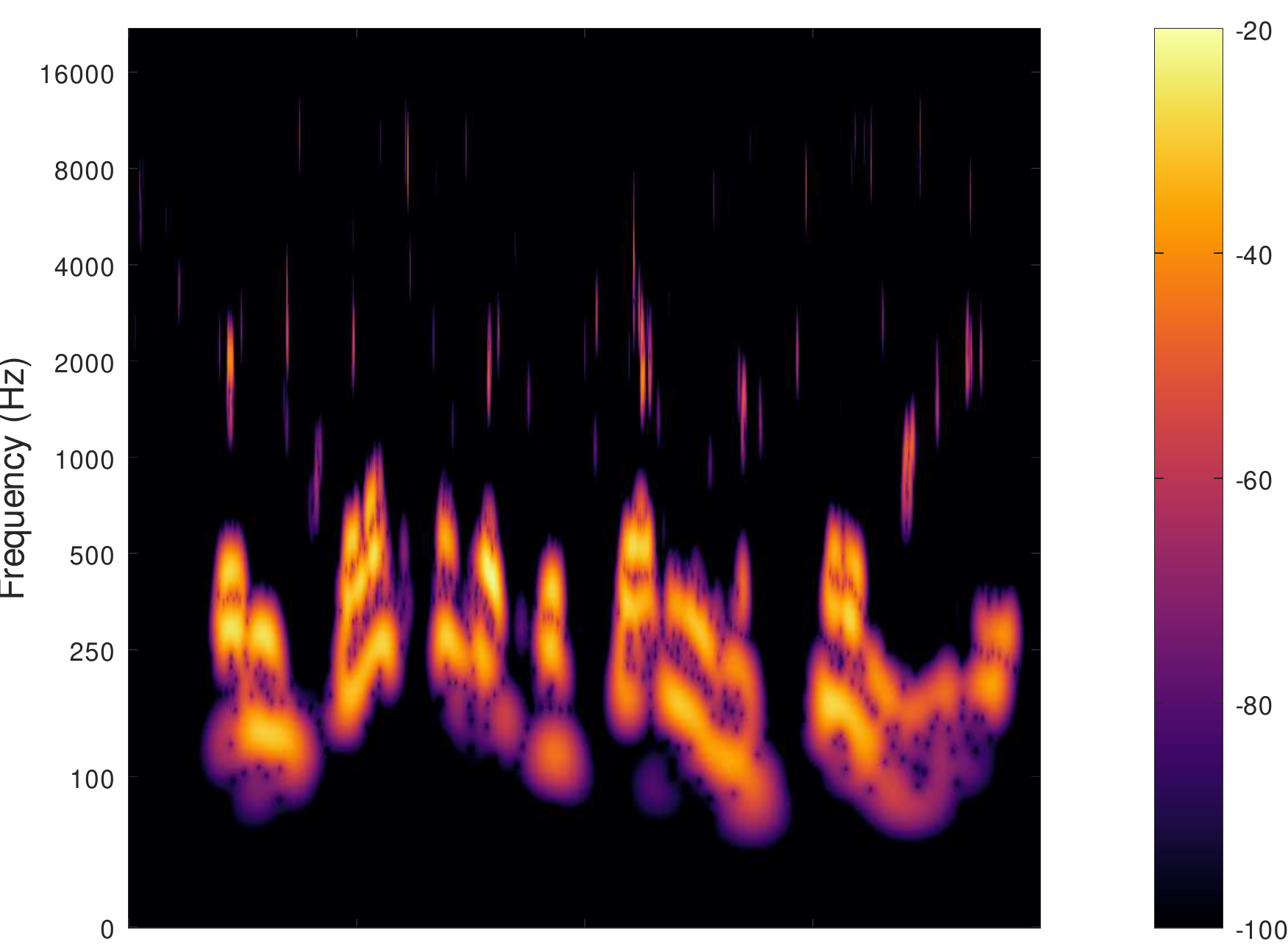}
\end{minipage}
\begin{minipage}[c]{0.45\textwidth}
\includegraphics[width=1\textwidth,trim={1.5cm 6.9cm 2.3cm 7.1cm},clip]{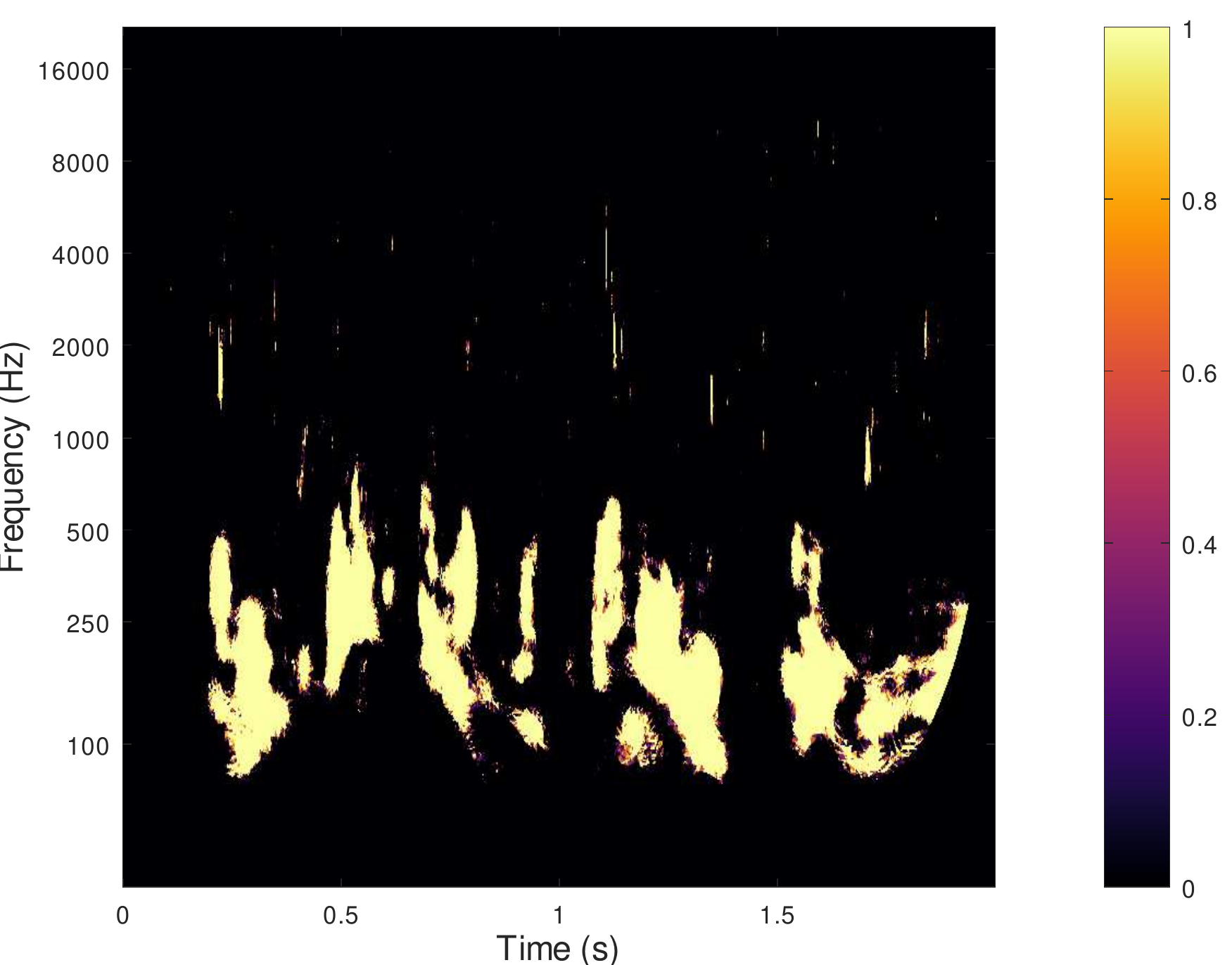}
\end{minipage}
\begin{minipage}[c]{0.45\textwidth}
\includegraphics[width=1\textwidth,trim={1.5cm 6.9cm 2.3cm 7.1cm},clip]{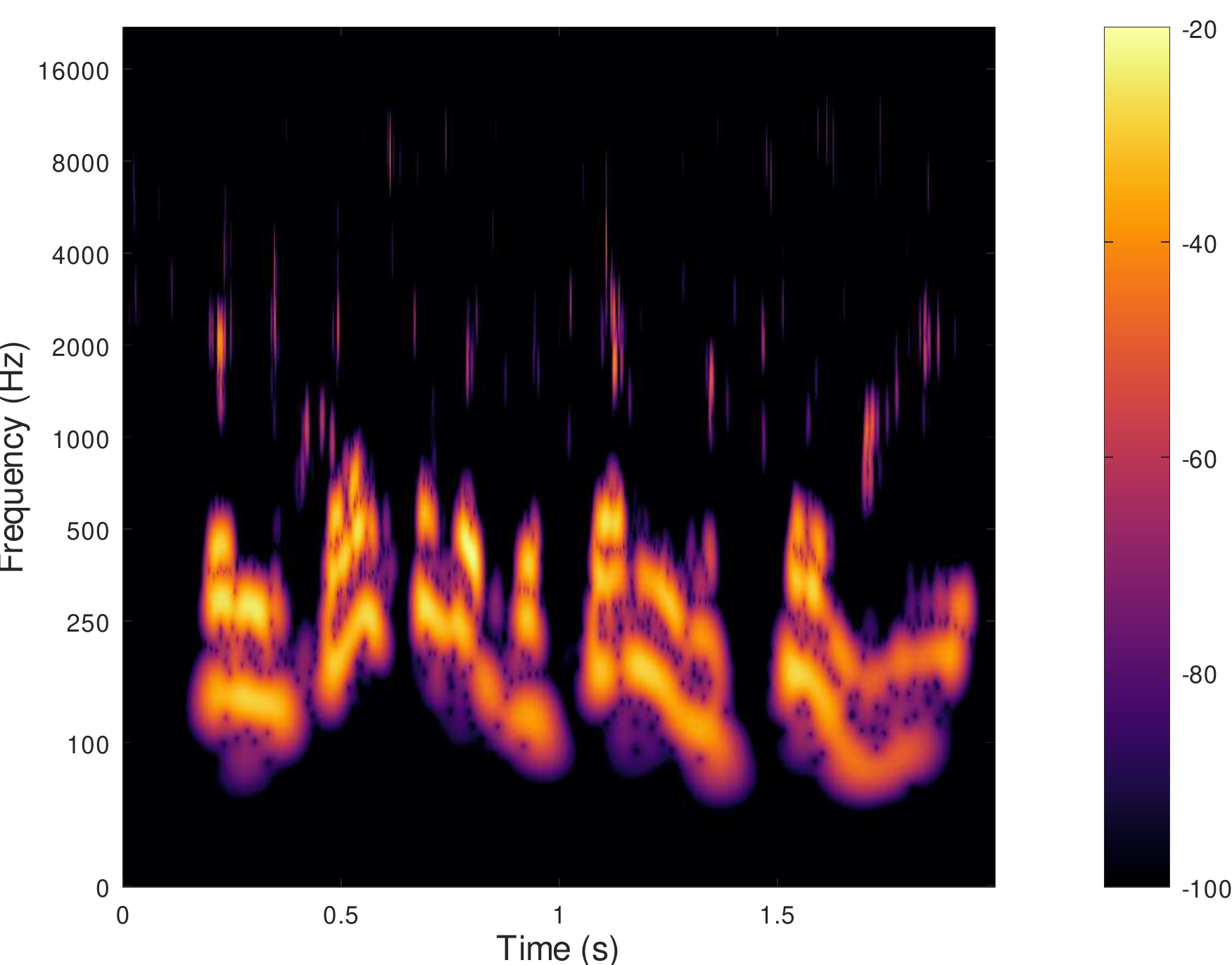}
\end{minipage}
\captionsetup{width=\linewidth}
\caption{Scalogram of a two second male voice signal: clean signal (top left), signal in white noise (top right), mask and scalogram of filtered signal based on: first intensity estimation (second row), pair correlation estimation (third row), and both estimators (bottom row).
All cases were additionally filtered to region of interest for comparison.}%\vspace{-6mm}}
\label{fig:voice}
\end{figure}
%In preliminary experiments, we observed that zeros are repelled by signal components and not clustered.
%Thus, we slightly change the mask suggested in \eqref{eq:maskint} to only emphasize areas where the first intensity function is lower than expected for white noise, i.e., we use
%\be\label{eq:maskintadap}
  %q(w) = \min\bigg\{
  %\max \bigg\{
  %a
  %\bigg(\sum_{k=1}^K\frac{ \big(  \max\{\mu_{r_k}- \card{Z \cap D_\mathrm{ph}(w,r_k)},0\}\big)^2}{K\sigma_{r_k}^2}
  %\bigg) 
  %-b
  %, 0 \bigg\}
  %, 1
  %\bigg\}.
%\ee
For the filtering scheme presented in \eqref{eq:maskint}, we chose $a=1$ and $b=4.32$.
This value of $b$ was obtained in simulations as the $0.999$ quantile of the calculated statistic for white noise and thus results in almost complete elimination of noise. 
Only the most prominent signal components are preserved.
For the filtering scheme presented in \eqref{eq:maskpc}, we chose $a_{k,\ell}=1/\hat{\sigma}^2_{r_{0,k},r_{1,\ell}}$ where $\hat{\sigma}_{r_{0,k},r_{1,\ell}}$ are the sample standard deviations from Table~\ref{tab:paircorrvar} that we obtained from white noise simulations.
We further chose $b=5.42$, again obtained in simulations as the $0.999$ quantile of the calculated statistic for white noise.
Furthermore, as above, we excluded parts of the scalogram where hyperbolic circles did no longer fit into the observation window and thus estimators would have to deal with non-negligible boundary effects. 
This was simply done by filtering all signals with a mask that is zero outside of this ``region of interest'' and equal to one inside.
It is difficult to interpret the masks found by our estimators.
For the first intensity function, we observe that the mask primarily covers areas with no zeros in the vicinity and some regions with a too high density of zeros between those regions.
For the mask generated by estimating the pair correlation function there is no obvious interpretation. 
Nevertheless, the estimated pair correlation function finds areas where the point pattern deviates from the pattern expected for white noise and thus we expect that a combination of the masks will result in better performance. 
The final mask in Fig.~\ref{fig:voice} illustrates an approach where we first added the deviations for both estimators, and applied thresholding afterwards.
The threshold was again determined as the simulated $0.999$ quantile of the corresponding statistic in the white noise case.

As a second example we present the castanets signal~27 in~\cite{sqam} from $0.4$s to $2.4$s. 
We expected that for this percussion signal that has properties very similar to white noise, our filtering framework will perform poorly. 
However, as one can see in Fig.~\ref{fig:cast}, the main onsets of each transient is actually identified correctly. 
On the other hand, the noise-like signals between the transients  are filtered out as expected.
\begin{figure}[tbp]%
\begin{minipage}[c]{0.45\textwidth}
\includegraphics[width=1\textwidth,trim={1.5cm 7.8cm 2.3cm 7.1cm},clip]{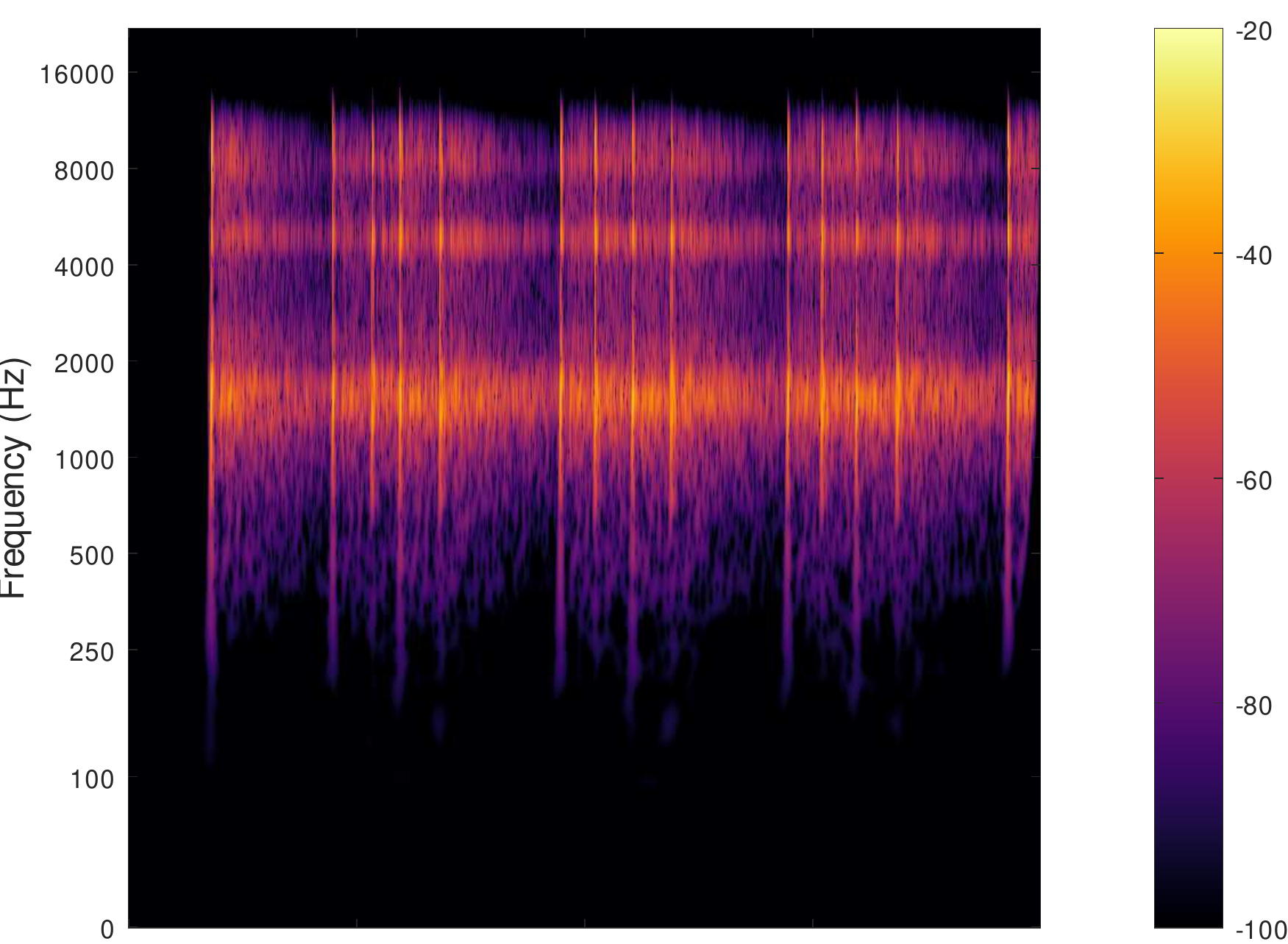}
\end{minipage}
\begin{minipage}[c]{0.45\textwidth}
\includegraphics[width=1\textwidth,trim={1.5cm 7.8cm 2.3cm 7.1cm},clip]{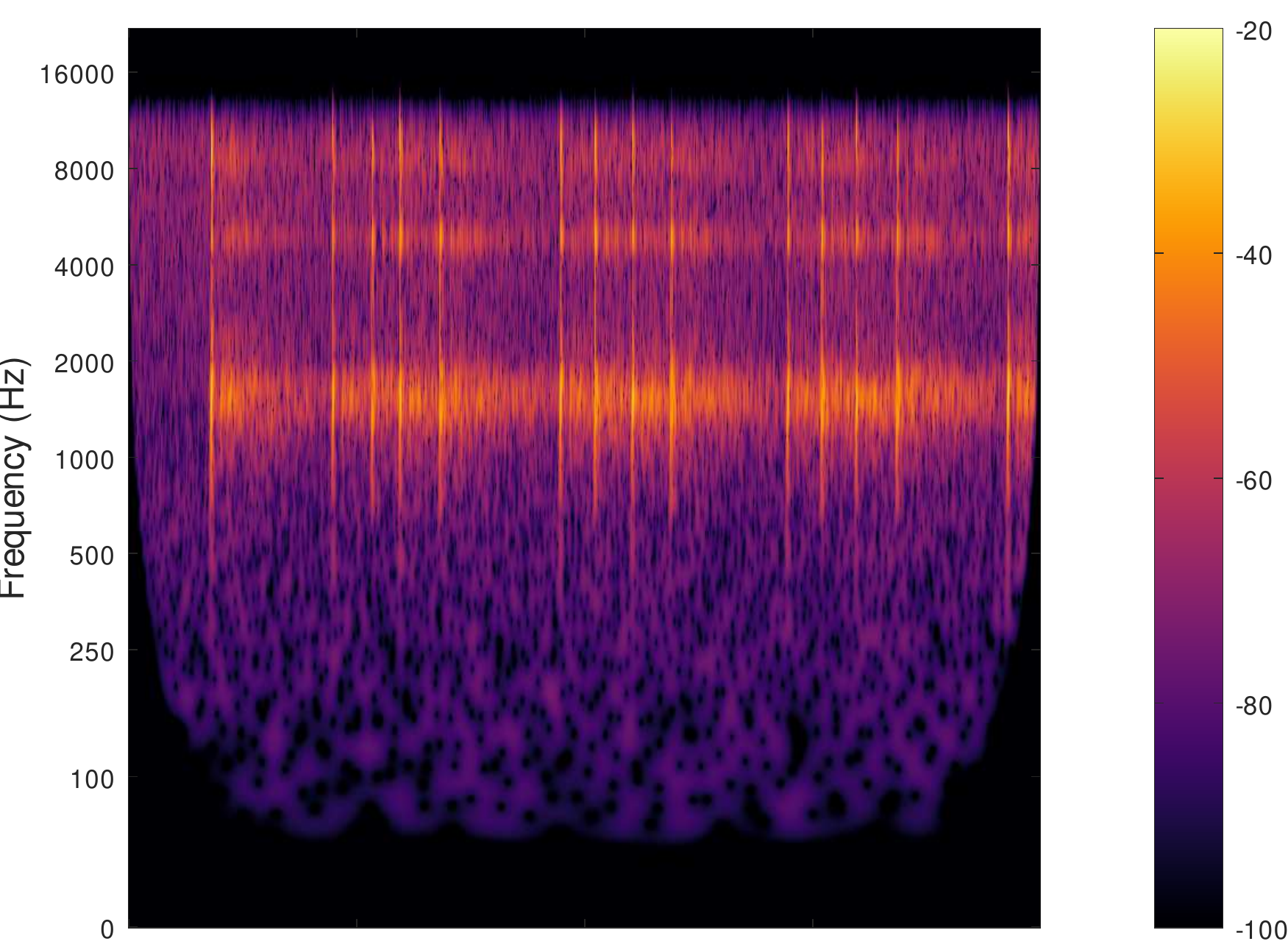}
\end{minipage}
\begin{minipage}[c]{0.45\textwidth}
\includegraphics[width=1\textwidth,trim={1.5cm 7.8cm 2.3cm 7.1cm},clip]{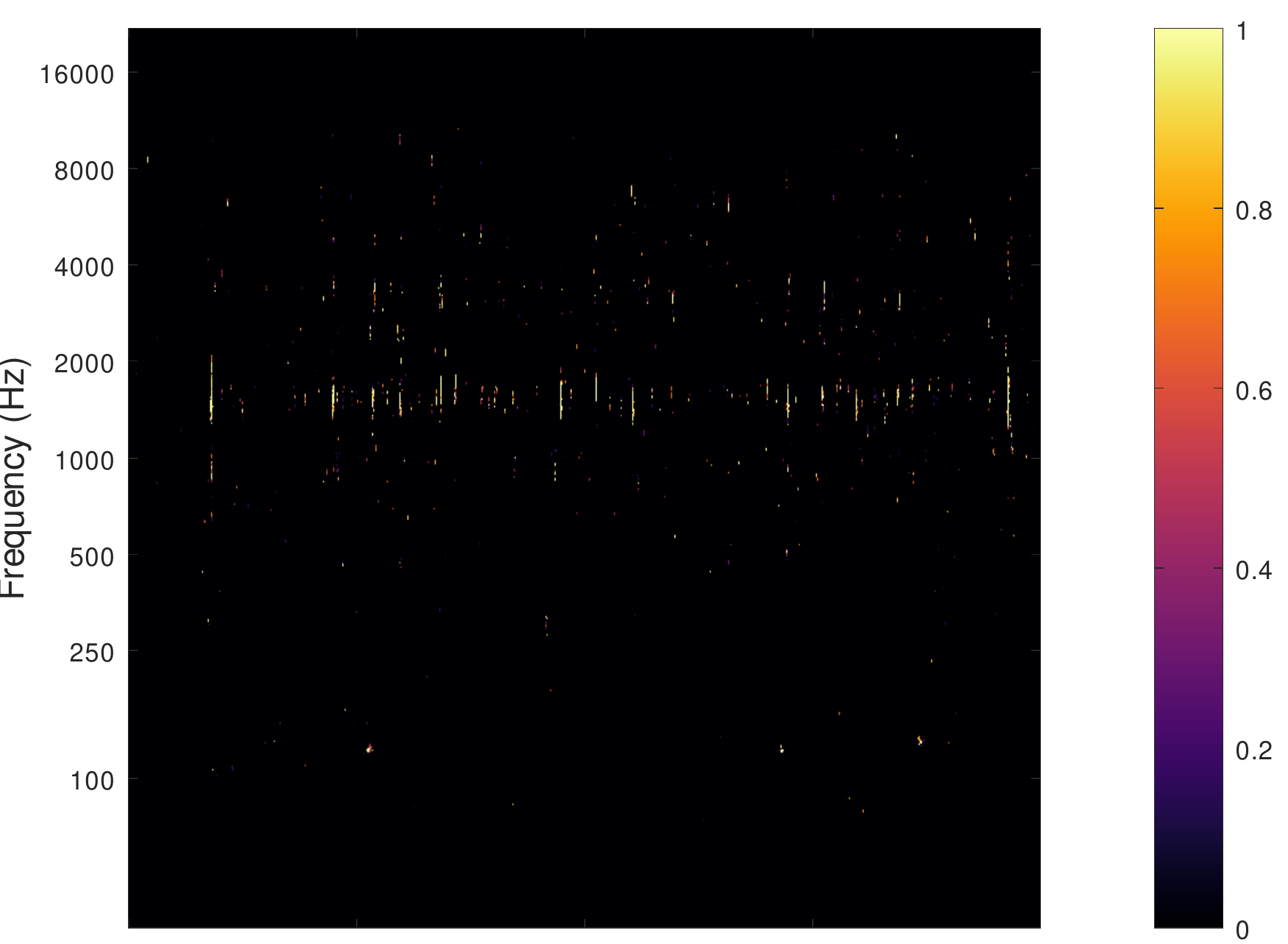}
\end{minipage}
\begin{minipage}[c]{0.45\textwidth}
\includegraphics[width=1\textwidth,trim={1.5cm 7.8cm 2.3cm 7.1cm},clip]{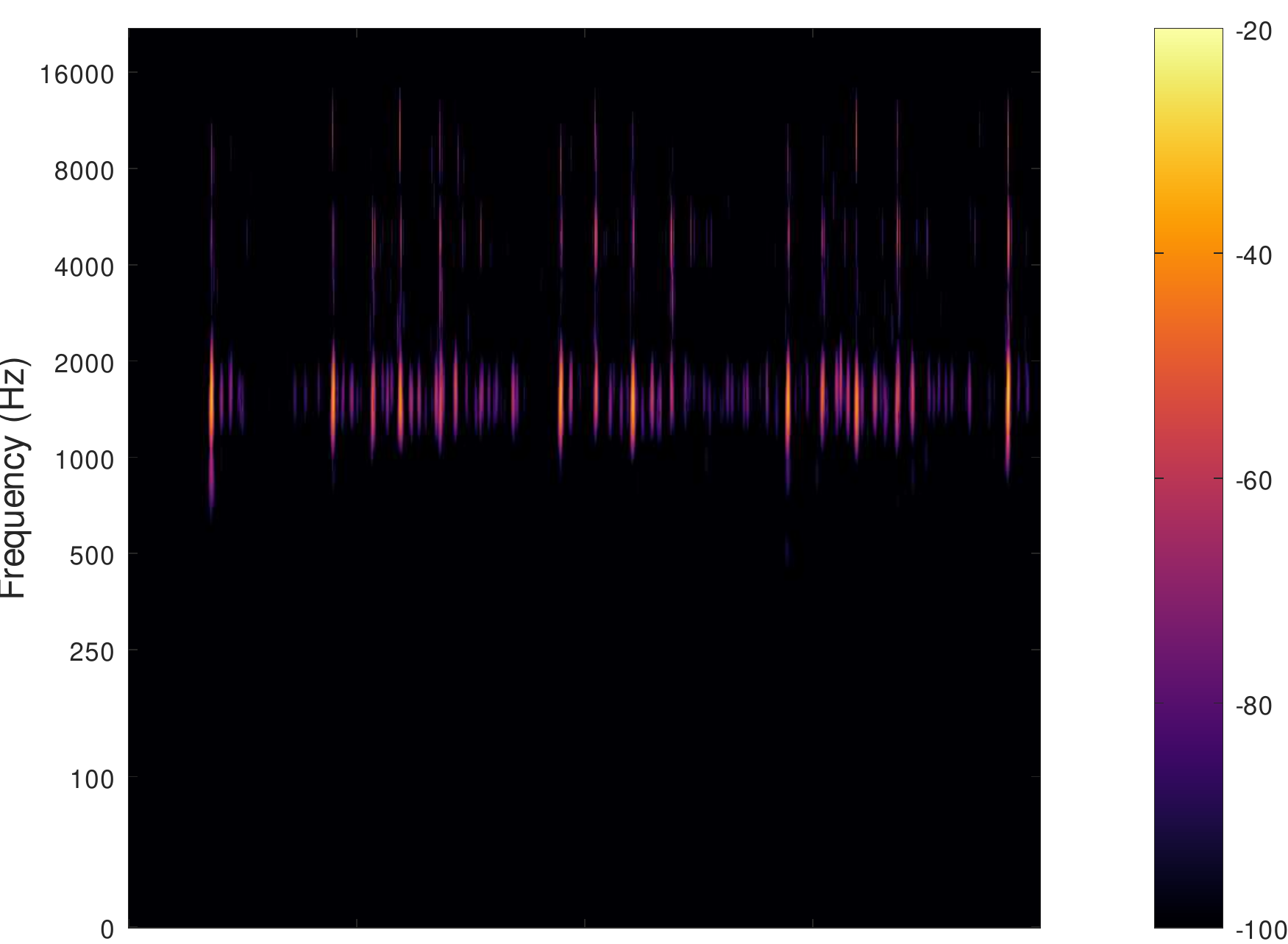}
\end{minipage}
\begin{minipage}[c]{0.45\textwidth}
\includegraphics[width=1\textwidth,trim={1.5cm 7.8cm 2.3cm 7.1cm},clip]{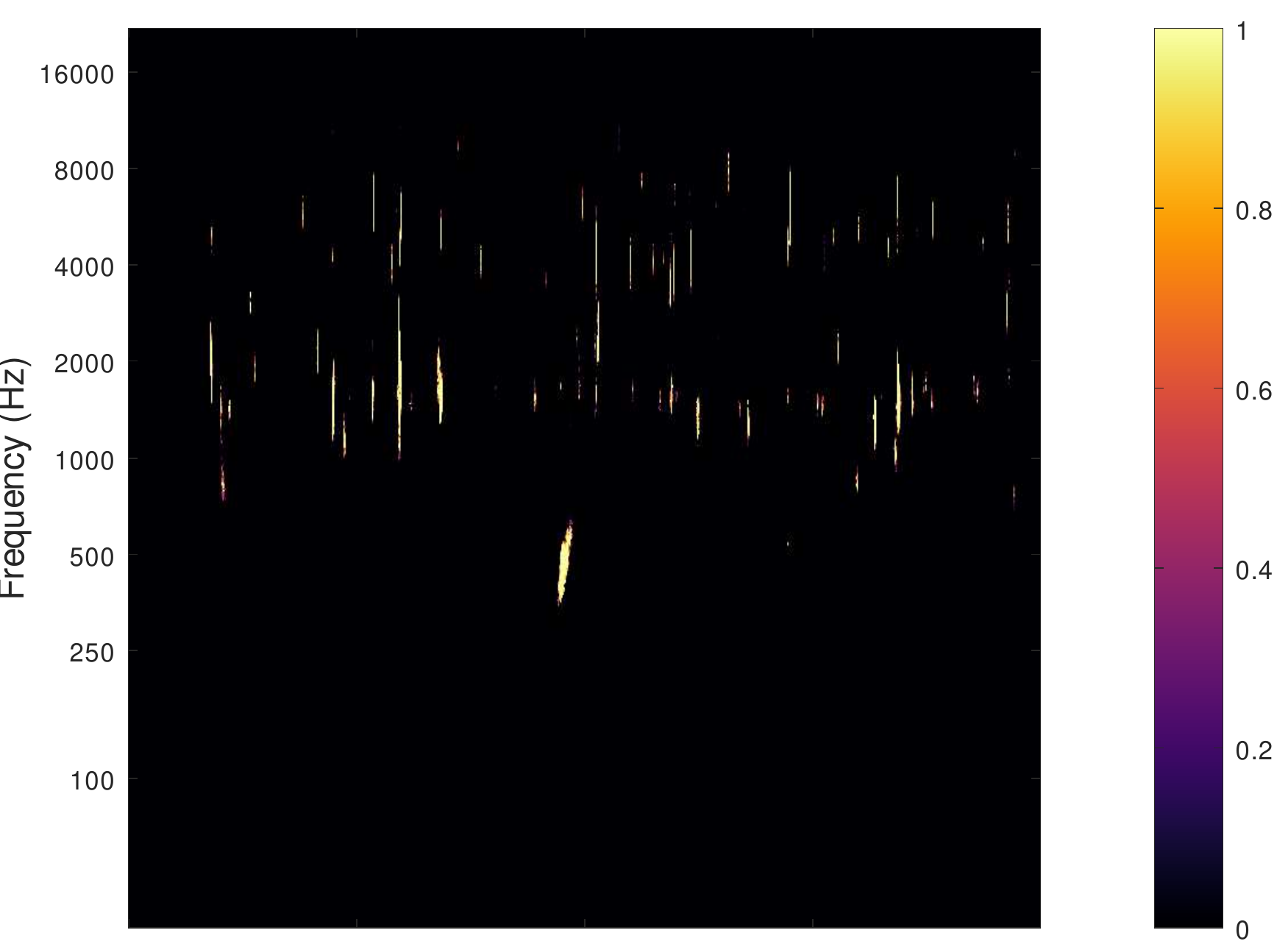}
\end{minipage}
\begin{minipage}[c]{0.45\textwidth}
\includegraphics[width=1\textwidth,trim={1.5cm 7.8cm 2.3cm 7.1cm},clip]{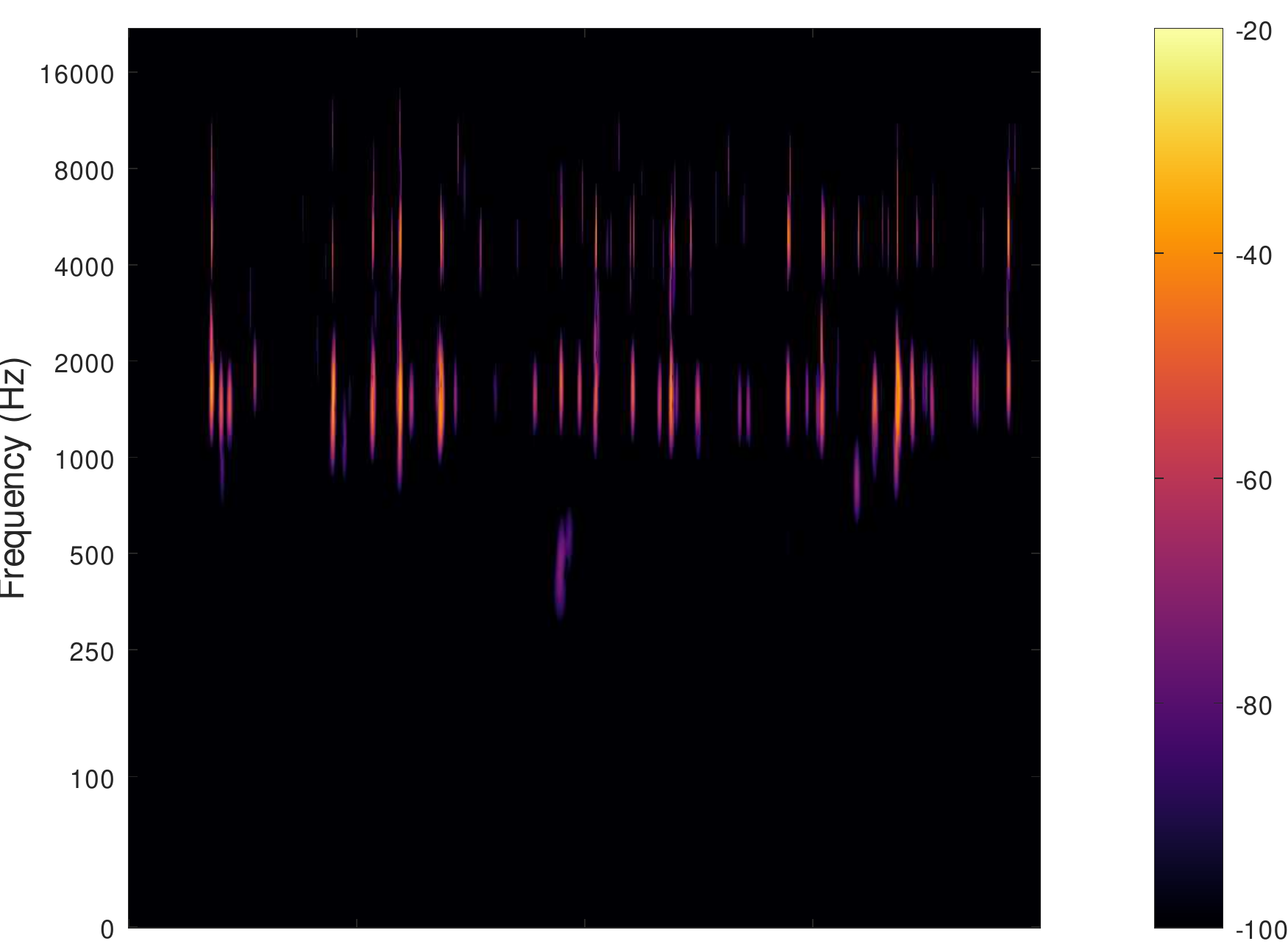}
\end{minipage}
\begin{minipage}[c]{0.45\textwidth}
\includegraphics[width=1\textwidth,trim={1.5cm 6.9cm 2.3cm 7.1cm},clip]{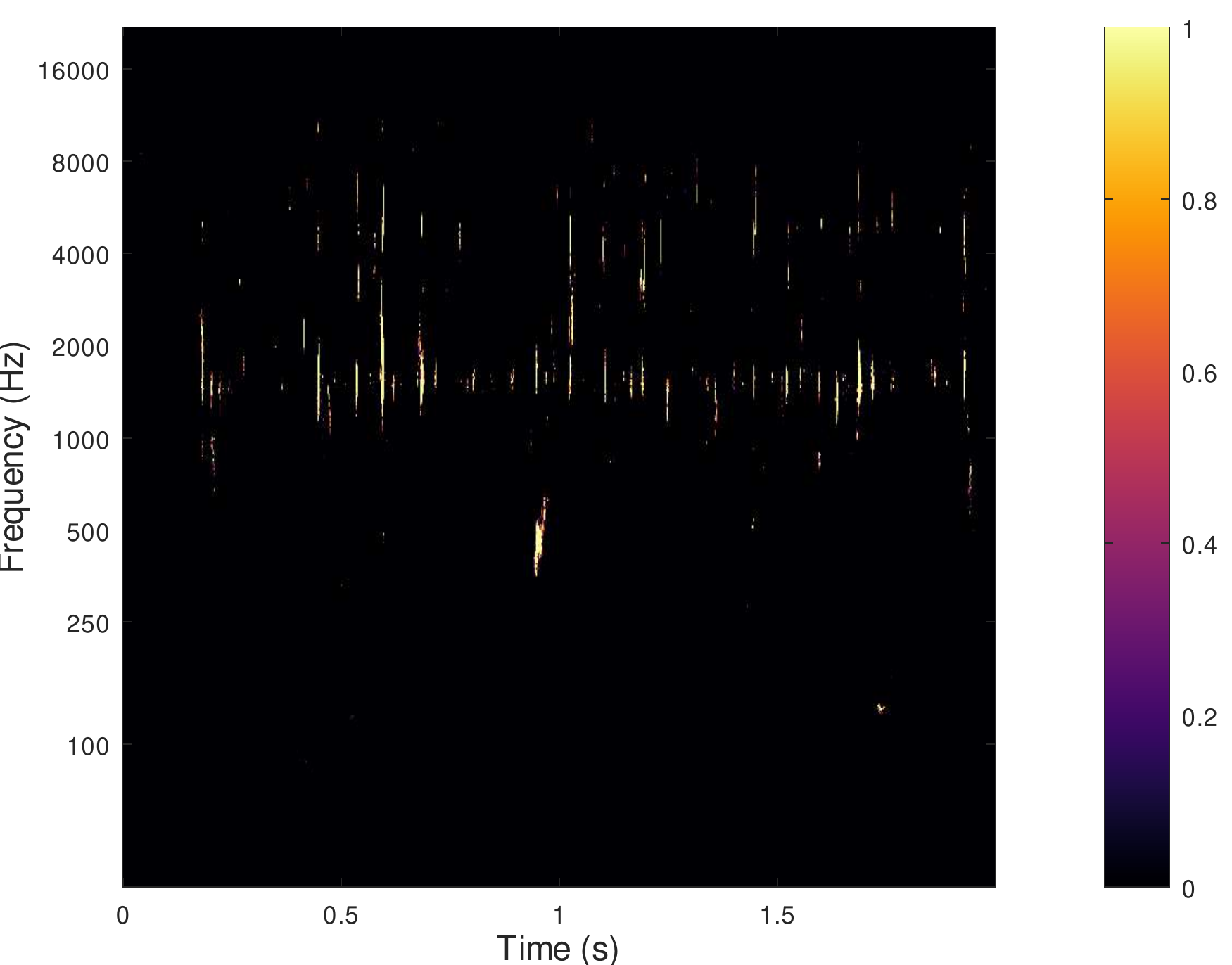}
\end{minipage}
\begin{minipage}[c]{0.45\textwidth}
\includegraphics[width=1\textwidth,trim={1.5cm 6.9cm 2.3cm 7.1cm},clip]{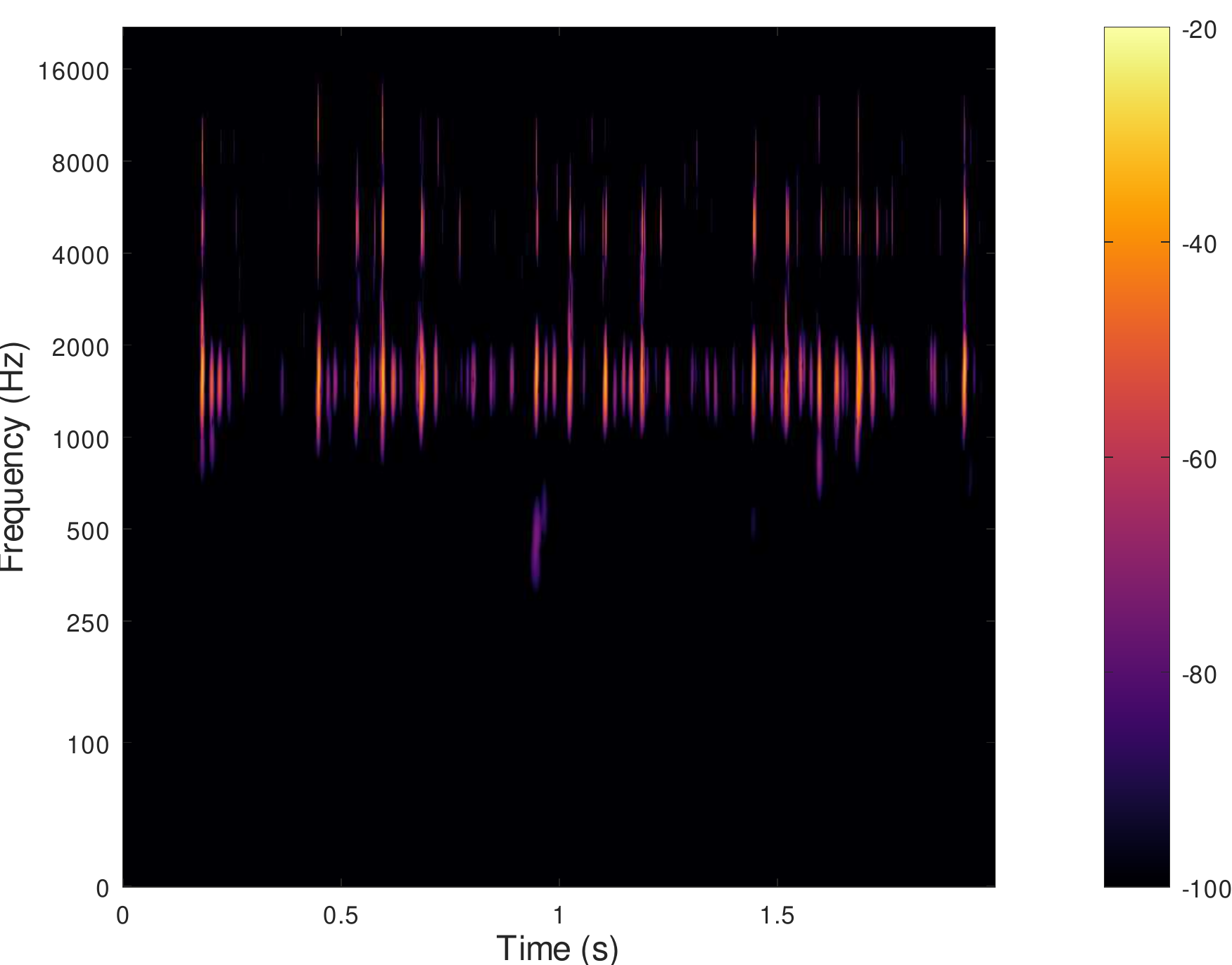}
\end{minipage}
\captionsetup{width=\linewidth}
\caption{Scalogram of a two second castanets signal: clean signal (top left), signal in white noise (top right), mask and scalogram of filtered signal based on: first intensity estimation (second row), pair correlation estimation (third row), and both estimators (bottom row).
All cases were additionally filtered to region of interest for comparison.}%\vspace{-6mm}}
\label{fig:cast}
\end{figure}

We emphasize that in all cases the filtering procedure is only based on the zeros of the scalogram, i.e., the modulus at other points was not known to the algorithm while generating the filter masks. 
This is in stark contrast to the common setting of thresholding where only the high-energy components are used to design the mask and we thus expect that these complementary approaches can be combined for a richer analysis and filtering of signals. 

We also applied the algorithm to a white noise signal to see if false positives appear in this setting. 
The result is presented in Fig.~\ref{fig:noise}.
\begin{figure}%
\begin{minipage}[c]{0.48\textwidth}
\includegraphics[width=1\textwidth,trim={1.5cm 7.8cm 2.3cm 7.1cm},clip]{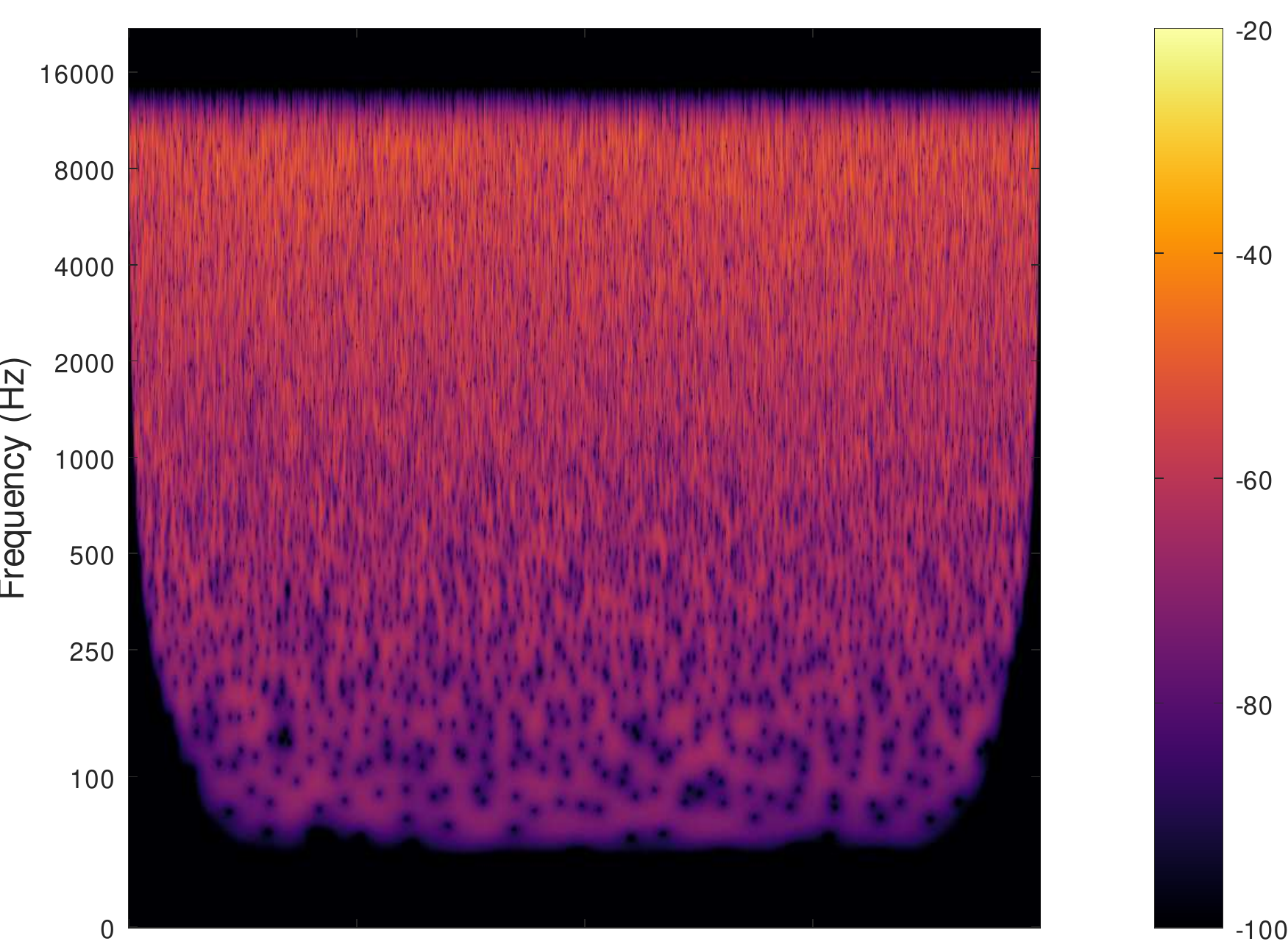}
\end{minipage}
\begin{minipage}[c]{0.48\textwidth}
\includegraphics[width=1\textwidth,trim={1.5cm 7.8cm 2.3cm 7.1cm},clip]{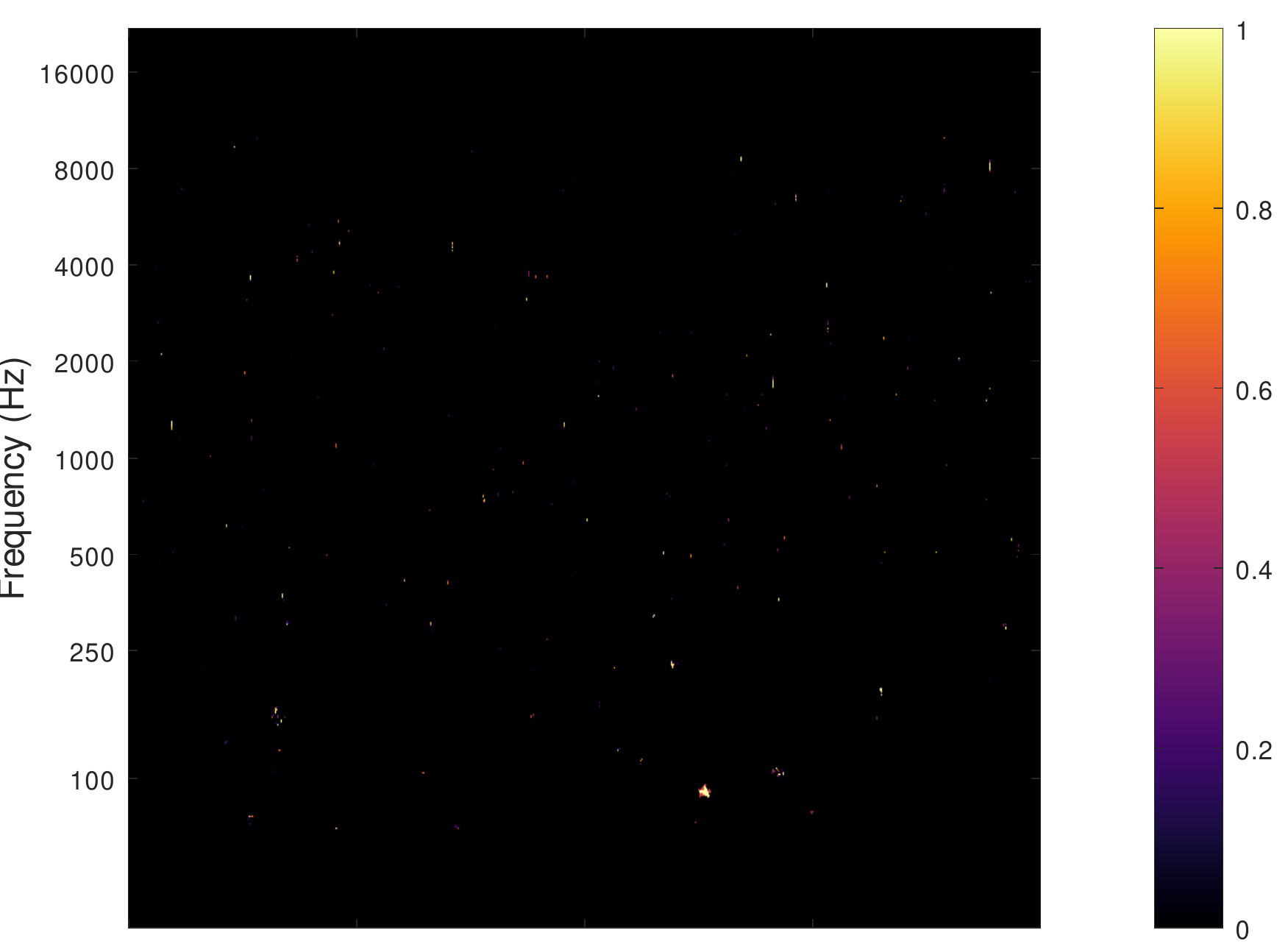}
\end{minipage}
\begin{minipage}[c]{0.48\textwidth}
\includegraphics[width=1\textwidth,trim={1.5cm 6.9cm 2.3cm 7.1cm},clip]{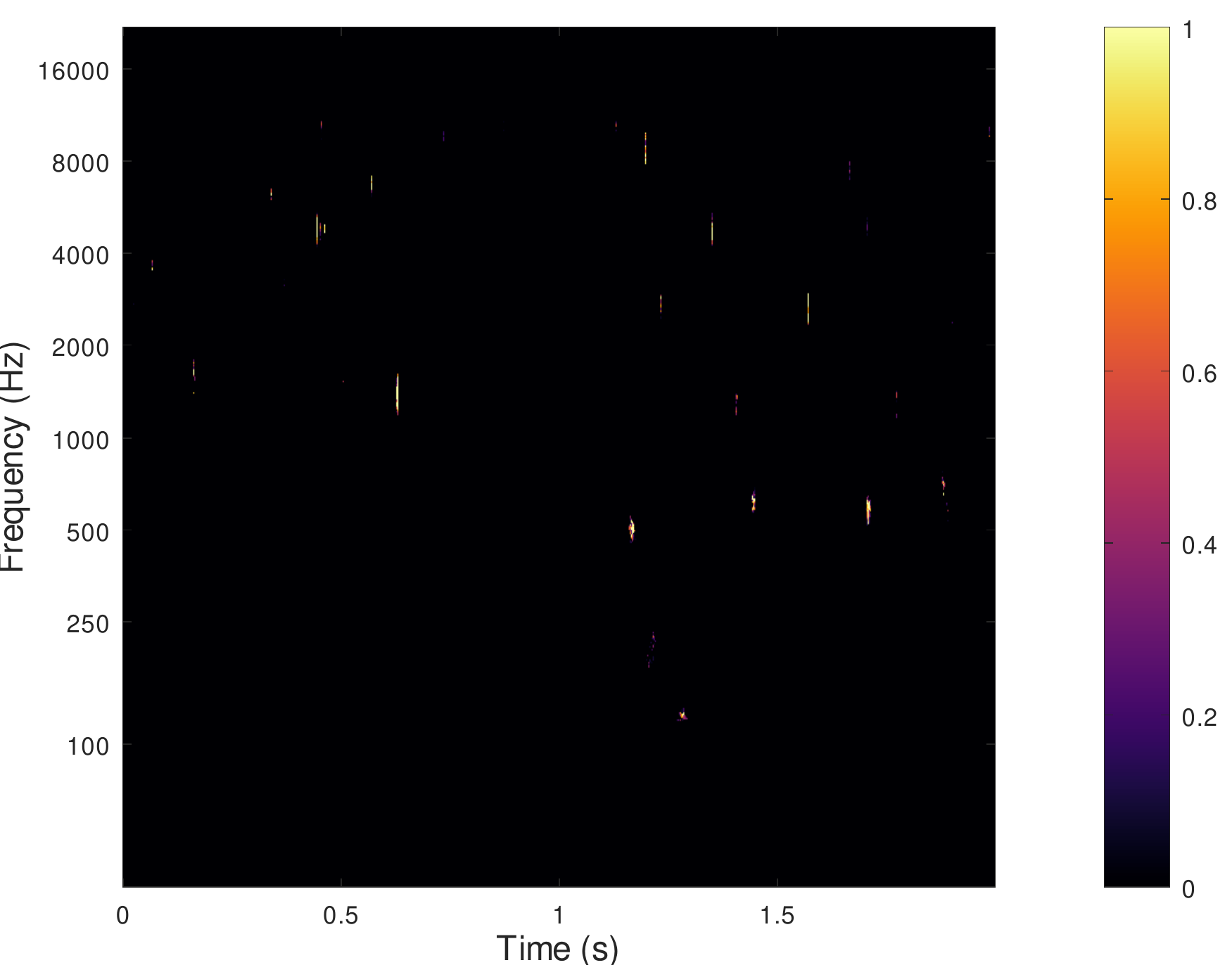}
\end{minipage}
\begin{minipage}[c]{0.48\textwidth}
\includegraphics[width=1\textwidth,trim={1.5cm 6.9cm 2.3cm 7.1cm},clip]{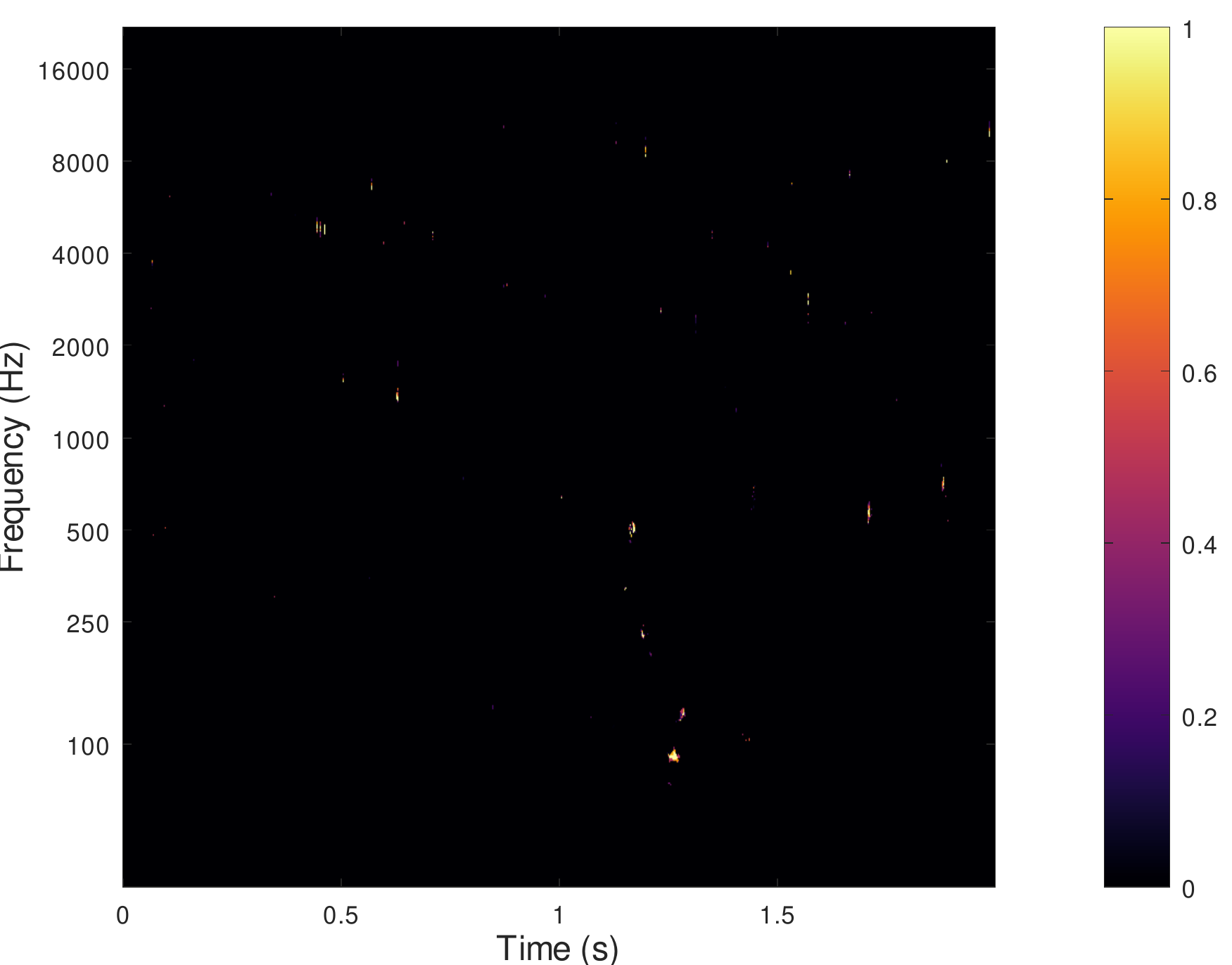}
\end{minipage}
\caption{Scalogram of a filtered two second white noise signal: white noise (top left), mask based on first intensity estimation (top right), mask  based on pair correlation estimation (bottom left), and mask based on both estimators (bottom right).
All cases were filtered to region of interest for comparison.
As expected by the choice of our thresholds almost all of the noise is filtered out.}%\vspace{-6mm}}
\label{fig:noise}
\end{figure}
As expected, we observe that there are hardly any false positives in the observation window for the parameters we chose.

%%%%%%%%%%%%%%%%%%%%%%%%%%%%%%%%%%%%%%%%%
\section{Conclusion}\label{sec:conclusion}
%%%%%%%%%%%%%%%%%%%%%%%%%%%%%%%%%%%%%%%%%

We presented an approach to identify signal components based only on the zero distribution of the continuous wavelet transform.
The basic idea for this method was proposed in \cite{flandrin2015time} for the short time frequency transform setting.
Using analyticity inducing wavelets, we used the relation to hyperbolic GAFs to obtain a basic understanding of the distribution of the zero pattern of wavelet transformed white noise.
This relation allowed us to use established results for the first intensity function and the pair correlation function to design filtering procedures based on local estimators of these quantities.
Furthermore, next to simulations also an asymptotic result ensures that the continuous setting can be well approximated in the discrete setting and thus the insights obtained from the continuous regime can be used in a realistic scenario.

Many questions regarding optimal choices of parameters and thresholds can be considered in future work.
For example, the exact role of the parameter $\alpha$ is largely unknown and so far we only observed in limited experiments that the parameter $\alpha$ manages a tradeoff between time and scale resolution.
Furthermore, the masks can obviously be adapted to various settings using, e.g., higher exponents in the difference calculation.
Finally, the approach in this work will complement classic methods that focus primarily on high energy components and thus seem almost orthogonal to our method. 
We expect that the identification of signal components in noise can be used in a multitude of applications way beyond the toy example of audio signals that was primarily chosen for illustrative reasons in this work.

\appendix

\section{Lemmata}\label{app:lem}

\begin{lemma}\label{lem:variance}
  Let $F(z)=W^{(d)}_{\mwlet_{\alpha }}(\mathcal{N}_{L,T_s})(z) - W_{\mwlet_{\alpha }}(\mathcal{N})(z)$
  where $W_{\mwlet_{\alpha }}(\mathcal{N})$ is given in \eqref{eq:wavelet_wn} and $W^{(d)}_{\mwlet_{\alpha }}(\mathcal{N}_{L,T_s})$ in \eqref{eq:wtdisc}.
  Furthermore, let  $T_s=L^{-\frac{\alpha}{\alpha+2}}$.
  Then the maximal variance $\sigma_{D,L}^2 = \sup_{z \in D} \E[(F_L(z))^2]$ on a compact set $D\subseteq \Pi^+$ is bounded by
  \be
    \sigma_{D,L}^2 \leq \bigg(2 \Gamma^2\bigg(\frac{\alpha+3}{2}\bigg) \frac{1}{y_{\textrm{min}}^2}
    + 2 \Gamma^2\bigg(\frac{\alpha+1}{2}\bigg)   \frac{(2y_{\textrm{max}})^{\alpha-1}}{ (\alpha-1) } \bigg) L^{-2\frac{\alpha-1}{\alpha+2}} 
  \ee
  for all $L \geq (2 x_{\textrm{max}})^{\frac{\alpha+2}{2}}$ where  $x_{\textrm{max}}:=   \sup_{u+iv\in D} \abs{u}$, $y_{\textrm{max}}:= \sup_{u+iv\in D}v$,  and $y_{\textrm{min}}:= \inf_{u+iv\in D}v$.
\end{lemma}
\begin{proof}
  Based on the definition of $\mathcal{N}$   the variance of $F_L(z)$ is given by
  \be
    \sigma_{z,L}^2 = \Biggnorm{\transl_x \dil_y \mwlet_{\alpha }
    - \sum_{\ell=-L+1}^L \transl_x \dil_y \mwlet_{\alpha }(\ell T_s)\ind_{((\ell-1)T_s, \ell T_s) }
    }_{\LtR}^2\,.
  \ee
  To bound $\sigma_{z,L}^2$ we first prove that $\mwlet_{\alpha}$ is Lipschitz continuous. 
  One way to see this is by bounding the derivative of $\mwlet_{\alpha}$.
  Due to the characterization of $\mwlet_{\alpha}$ by its Fourier transform in \eqref{eq_ga} and using the fact that taking the derivative in time-domain corresponds to multiplication by the frequency variable in frequency-domain (up to a multiplication by $i$), we see that $\mwlet_{\alpha}'(t) = i \mwlet_{\alpha+2}(t)$. 
  We furthermore have
  \be
    \abs{\mwlet_{\alpha+2}(t)} 
    = \biggabs{\frac{\Gamma(\frac{\alpha+3}{2})}{(1-i t)^{\frac{\alpha+3}{2}} }} 
    \leq \Gamma\bigg(\frac{\alpha+3}{2}\bigg),
  \ee
  i.e., the derivative is bounded and in turn $\mwlet_{\alpha}$ Lipschitz continuous with Lipschitz constant $\Gamma\big(\frac{\alpha+3}{2}\big)$.
  Now we can decompose $\sigma_z^2$ into a tail and central part.
  More specifically, we have
  \ba
    \sigma_{z,L}^2 
    & = 
    \int_{-L T_s}^{LT_s} 
      \bigg(\transl_x \dil_y \mwlet_{\alpha }(t)
    - \sum_{\ell=-L+1}^L \transl_x \dil_y \mwlet_{\alpha }(\ell T_s)\ind_{((\ell-1)T_s, \ell T_s) }(t)\bigg)^2 \, dt
    \notag \\
    & \quad 
    + \int_{\abs{t}>LT_s} 
      \big(\transl_x \dil_y \mwlet_{\alpha }(t)\big)^2 \, dt
    \notag \\
    & \leq 
    \int_{-L T_s}^{LT_s} 
      \Gamma^2\bigg(\frac{\alpha+3}{2}\bigg)\frac{T_s^2}{y^2} \, dt
      + 2 \int_{\frac{LT_s-\abs{x}}{y}}^{\infty} \mwlet_{\alpha }^2(t) \, dt
    \notag \\
    & \leq 
    2 \Gamma^2\bigg(\frac{\alpha+3}{2}\bigg) \frac{L T_s^3}{y^2}
    + 2 \int_{\frac{LT_s-\abs{x}}{y}}^{\infty} \frac{\Gamma^2(\frac{\alpha+1}{2})}{(1+ t^2)^{\frac{\alpha+1}{2}} } \, dt
    \notag \\
    & \leq 
    2 \Gamma^2\bigg(\frac{\alpha+3}{2}\bigg) \frac{L T_s^3}{y^2}
    + 2 \Gamma^2\bigg(\frac{\alpha+1}{2}\bigg) \int_{\frac{LT_s-\abs{x}}{y}}^{\infty} \frac{1}{ t^{\alpha+1} } \, dt
    \notag \\
    & =
    2 \Gamma^2\bigg(\frac{\alpha+3}{2}\bigg) \frac{L T_s^3}{y^2}
    + 2 \Gamma^2\bigg(\frac{\alpha+1}{2}\bigg)   \frac{y^{\alpha-1}}{ (\alpha-1)(LT_s-\abs{x})^{\alpha-1} }.
  \ea
  To bound this independently of $z$, we note that $y_{\textrm{min}} \leq y \leq y_{\textrm{max}}$.
  Furthermore, recall that $T_s=L^{-\frac{\alpha}{\alpha+2}}$ and thus our assumption $L\geq  (2 x_{\textrm{max}})^{\frac{\alpha+2}{2}}$ is equivalent to  $ x_{\textrm{max}} \leq LT_s/2$.
  Hence, noting that $L T_s^3 = (LT_s)^{-\alpha+1} = L^{-2\frac{\alpha-1}{\alpha+2}}$, we can further bound independently of $z\in D$
  \ba
    \sigma_{D,L}^2 
    & \leq
    2 \Gamma^2\bigg(\frac{\alpha+3}{2}\bigg) \frac{L T_s^3}{y_{\textrm{min}}^2}
    + 2 \Gamma^2\bigg(\frac{\alpha+1}{2}\bigg)   \frac{(2y_{\textrm{max}})^{\alpha-1}}{ (\alpha-1)(LT_s)^{\alpha-1} } 
    \notag \\
    & =
    \bigg(2 \Gamma^2\bigg(\frac{\alpha+3}{2}\bigg) \frac{1}{y_{\textrm{min}}^2}
    + 2 \Gamma^2\bigg(\frac{\alpha+1}{2}\bigg)   \frac{(2y_{\textrm{max}})^{\alpha-1}}{ (\alpha-1) } \bigg) L^{-2\frac{\alpha-1}{\alpha+2}}\,.
    \notag 
  \ea
\end{proof}

\begin{lemma}\label{lem:bounddl}
  Let $d_L(w,z) = \big( \E \big[ (F(w)- F(z) )^2 \big] \big)^{1/2}$ with $F(z)$ as defined in Lemma~\ref{lem:variance}.
  Furthermore, let  $T_s=L^{-\frac{\alpha}{\alpha+2}}$ and $LT_s = L^{\frac{2}{\alpha+2}} > 1$.
  Then on any compact set $D\subseteq \Pi^+$
  \ba
    d_L(w,z) %\leq C L^{3/2}T_s^{3/2} \abs{w-z}
    & \leq 
    \bigg(
    \frac{ \Gamma\big(\frac{\alpha+3}{2}\big) }{ y_{\textrm{min}}^2} 
    \big( 
     2 +  4 (x_{\textrm{max}}^2 + y_{\textrm{max}}^2) 
    \big)^{\frac{1}{2}}
    \notag \\*
    & \qquad 
    + \bigg(\frac{3\Gamma(\alpha +2)}{2^{\alpha+2} y_{\textrm{min}}}
    +  \frac{ 2 \Gamma(\alpha) }{(2y_{\textrm{min}})^{\alpha}}
    \bigg)^{\frac{1}{2}}
    \bigg) L^{ \frac{3}{ \alpha+2 }} \abs{w-z} 
  \ea
  for all $w,z \in D$, 
  where  $x_{\textrm{max}}:=   \sup_{u+iv\in D} \abs{u}$, $y_{\textrm{max}}:= \sup_{u+iv\in D}v$,  and $y_{\textrm{min}}:= \inf_{u+iv\in D}v$.
\end{lemma}
\begin{proof}
 As a first step, we decompose $d_L(w,z)$ by the triangle inequality into
  \ba
    d_L(w,z)  & \leq \E \Big[ \bigabs{W_{\mwlet_{\alpha }}(\mathcal{N})(w)-   W_{\mwlet_{\alpha }}(\mathcal{N})(z) }^2 \Big]^{1/2}
    \notag \\
    & \quad +   
    \E \Big[  \bigabs{ W^{(d)}_{\mwlet_{\alpha }}(\mathcal{N}_{L,T_s})(w)- W^{(d)}_{\mwlet_{\alpha }}(\mathcal{N}_{L,T_s})(z)}^2 \Big]^{1/2}.
    \label{eq:decompd}
  \ea
  The second expectation in \eqref{eq:decompd} is bounded as
  \ba
    & \E \Big[  \bigabs{ W^{(d)}_{\mwlet_{\alpha }}(\mathcal{N}_{L,T_s})(w)- W^{(d)}_{\mwlet_{\alpha }}(\mathcal{N}_{L,T_s})(z) }^2 \Big]
    \notag \\
    & =
    \E \Bigg[ \biggabs{\sum_{\ell=-L+1}^L \mathcal{N}_{L,T_s}[\ell] \transl_u \dil_v \mwlet_{\alpha }(\ell T_s)
    - \mathcal{N}_{L,T_s}[\ell] \transl_x \dil_y \mwlet_{\alpha }(\ell T_s)
    }^2 \Bigg]
    \notag \\
    & =
    T_s \bigg(\sum_{\ell=-L+1}^L   \abs{\transl_u \dil_v \mwlet_{\alpha }(\ell T_s)}^2
    + \abs{\transl_x \dil_y \mwlet_{\alpha }(\ell T_s)}^2
    \notag \\*
    & \qquad - 2 \operatorname{Re} \big(\transl_u \dil_v \mwlet_{\alpha }(\ell T_s)\transl_x \dil_y \mwlet_{\alpha }(\ell T_s)\big)
    \bigg) 
    \notag \\
    & =
    T_s \bigg(\sum_{\ell=-L+1}^L   \abs{\transl_u \dil_v \mwlet_{\alpha }(\ell T_s) 
    - \transl_x \dil_y \mwlet_{\alpha }(\ell T_s)}^2
    \bigg) 
    \notag \\
    & \leq 
    T_s \Gamma^2\bigg(\frac{\alpha+3}{2}\bigg)\Bigg(\sum_{\ell=-L+1}^L   \bigg(
    \frac{\ell T_s -u}{v} - \frac{\ell T_s -x}{y}
    \bigg)^2
    \Bigg) 
    \notag \\
    & \leq 
    T_s \Gamma^2\bigg(\frac{\alpha+3}{2}\bigg)\Bigg(\sum_{\ell=-L+1}^L 
    \frac{2 \ell^2 T_s^2 (y-v)^2 + 2 (xv-uy)^2}{v^2y^2}
    \Bigg) 
    \notag \\
    & \leq 
    \frac{  L T_s}{v^2y^2} \Gamma^2\bigg(\frac{\alpha+3}{2}\bigg)
    \big( 
     2 L^2 T_s^2 (y-v)^2 +   \abs{z\bar{w}-\bar{z}w}^2 
    \big) 
    \notag \\
    & = 
    \frac{  L T_s}{v^2y^2} \Gamma^2\bigg(\frac{\alpha+3}{2}\bigg)
    \big( 
     2 L^2 T_s^2 (y-v)^2 +   \abs{z\bar{w}-z\bar{z}+ z\bar{z} - \bar{z}w}^2 
    \big) 
    \notag \\
    & \leq 
    \frac{  L T_s}{v^2y^2} \Gamma^2\bigg(\frac{\alpha+3}{2}\bigg)
    \big( 
     2 L^2 T_s^2 (y-v)^2 +  4\abs{z}^2 \abs{w-z}^2
    \big) 
    \notag \\
    & \leq 
    \frac{ \Gamma^2\big(\frac{\alpha+3}{2}\big) }{ y_{\textrm{min}}^4} 
    \big( 
     2 +  4 (x_{\textrm{max}}^2 + y_{\textrm{max}}^2) 
    \big)  L^{\frac{6}{\alpha+2}}  \abs{w-z}^2 \,.
    \label{eq:boundsecondexpec}
  \ea
  For the first summand in \eqref{eq:decompd}, we have
  \ba
    & \E \Big[ \bigabs{W_{\mwlet_{\alpha }}(\mathcal{N})(w)-   W_{\mwlet_{\alpha }}(\mathcal{N})(z) }^2 \Big]
    \notag \\
    & = 
    \int_{\R} \biggabs{\mwlet_{\alpha}\bigg(\frac{t-x}{y}\bigg) - \mwlet_{\alpha}\bigg(\frac{t-u}{v}\bigg)}^2 \, \mathrm{d}t
    \notag \\
    & =
    \int_{\R^+} \bigabs{
    ye^{-ix\xi}\widehat{\mwlet_{\alpha}}(y\xi) - ve^{-iu\xi}\widehat{\mwlet_{\alpha}}(v\xi) 
    }^2 \, \mathrm{d}\xi
    %\notag \\
    %& =
    %\int_{\R^+} \bigabs{
      %ye^{-i(x-u)\xi}(y\xi)^{\frac{\alpha-1}{2}}e^{-y \xi } 
      %- v(v\xi)^{\frac{\alpha-1}{2}}e^{-v \xi } 
      %}^2 \, d\xi
    \notag \\
    & =
    \int_{\R^+} \bigabs{
      y^{\frac{\alpha+1}{2}}e^{-i(x-u)\xi} \xi^{\frac{\alpha-1}{2}}e^{-y \xi } 
      - v^{\frac{\alpha+1}{2}}  \xi^{\frac{\alpha-1}{2}}e^{-v \xi } 
      }^2 \, \mathrm{d}\xi
    \notag \\
    & =
    \int_{\R^+} \bigabs{
      y^{\frac{\alpha+1}{2}}(\cos((x-u)\xi) + i\sin((x-u)\xi)) \xi^{\frac{\alpha-1}{2}}e^{-y \xi } 
      \notag \\*
      & \qquad 
      - v^{\frac{\alpha+1}{2}}  \xi^{\frac{\alpha-1}{2}}e^{-v \xi } 
      }^2 \, \mathrm{d}\xi\,.
    \notag
  \ea
  Here, we split the modulus squared into real and imaginary parts and bound them separately.
  For the imaginary part squared, we obtain
  \ba
    & \int_{\R^+} 
      y^{\alpha+1}\sin^2((x-u)\xi) \xi^{ \alpha-1 }e^{-2y \xi }  
        \, \mathrm{d}\xi
    \notag \\
    & \leq
    \frac{(x-u)^2}{2^{\alpha+1}} \int_{\R^+} (2y\xi)^{\alpha+1}e^{-2y \xi }\, \mathrm{d}\xi
    \notag \\
    & =
    \frac{(x-u)^2}{2^{\alpha+2} y} \Gamma(\alpha +2)
    \notag \\
    & \leq
    \frac{\Gamma(\alpha +2)}{2^{\alpha+2} y_{\textrm{min}}} \abs{w-z}^2\,.
    \label{eq:boundimp}
  \ea
  For the real part squared, we have
  \ba
    &   \int_{\R^+}  
      \big(
      y^{\frac{\alpha+1}{2}}\cos((x-u)\xi)   \xi^{\frac{\alpha-1}{2}}e^{-y \xi } 
      - v^{\frac{\alpha+1}{2}}  \xi^{\frac{\alpha-1}{2}}e^{-v \xi } 
      \big)^2 \, \mathrm{d}\xi
    \notag \\
    & \leq
    \int_{\R^+}  2
      \xi^{ \alpha-1 } \big(
      y^{\frac{\alpha+1}{2}}\cos((x-u)\xi)   e^{-y \xi } 
      -y^{\frac{\alpha+1}{2}}    e^{-y \xi } 
      \big)^2
      \, \mathrm{d}\xi 
    \notag \\*
    & \quad
    +  \int_{\R^+}  2
      \xi^{ \alpha-1 }\big(
      y^{\frac{\alpha+1}{2}}    e^{-y \xi } 
      - v^{\frac{\alpha+1}{2}}   e^{-v \xi } 
      \big)^2
      \, \mathrm{d}\xi 
    \notag \\
    & \leq
    \int_{\R^+}  2
      \xi^{ \alpha+1 } y^{ \alpha+1 }e^{-2y \xi } 
      (x-u)^2
      \, \mathrm{d}\xi 
    \notag \\*
    & \quad
    +  \int_{\R^+}  2
      \xi^{ \alpha-1 }\big(
      y^{ \alpha+1}    e^{-2y \xi }
      - 2 (yv)^{\frac{\alpha+1}{2}}   e^{-(y+v) \xi }   
      + v^{ \alpha+1 }   e^{-2v \xi } 
      \big) 
      \, \mathrm{d}\xi 
    \notag \\
    & \leq
    \frac{(x-u)^2}{2^{\alpha+1}y}\Gamma(\alpha+2)
    %\notag \\*
    %& \quad
    +  \Gamma(\alpha) \bigg(
      \frac{y}{2^{\alpha-1}} - \frac{4 (yv)^{\frac{\alpha+1}{2}}}{(y+v)^{\alpha}}
      + \frac{v}{2^{\alpha-1}}
    \bigg)
    \notag \\
    & =
    \frac{(x-u)^2}{2^{\alpha+1}y}\Gamma(\alpha+2)
    +  4 \Gamma(\alpha) \bigg(
      \frac{ (\frac{y+v}{2})^{\alpha+1} - (yv)^{\frac{\alpha+1}{2}}}{ (y+v)^{\alpha}}
    \bigg)
    \notag \\
    & \stackrel{(a)}\leq 
    \frac{(x-u)^2}{2^{\alpha+1}y}\Gamma(\alpha+2)
    +  4 \Gamma(\alpha) 
      \frac{\frac{y^{\alpha+1}}{2} +\frac{v^{\alpha+1}}{2} - (yv)^{\frac{\alpha+1}{2}}}{ (y+v)^{\alpha}}
    \notag \\
    & =
    \frac{(x-u)^2}{2^{\alpha+1}y}\Gamma(\alpha+2)
    +  4 \Gamma(\alpha) 
      \frac{\Big(\frac{y^{\frac{\alpha+1}{2}}}{\sqrt{2}} - \frac{v^{\frac{\alpha+1}{2}}}{\sqrt{2}}\Big)^2}{ (y+v)^{\alpha}}
    \notag \\
    & \leq
    \bigg(\frac{ \Gamma(\alpha+2)}{2^{\alpha+1}y_{\textrm{min}}}
    +  \frac{ 2 \Gamma(\alpha) }{(2y_{\textrm{min}})^{\alpha}}
    \bigg) \abs{w-z}^2
    \label{eq:boundrealp}
  \ea
  where we used in $(a)$ Jensen's inequality for the convex function $(\cdot)^{\alpha+1}$.
  
  Combining the bounds \eqref{eq:boundsecondexpec}, \eqref{eq:boundimp}, and \eqref{eq:boundrealp} and noting that $  L^{\frac{3}{\alpha+2}}>1$, we obtain that
  \ba
    d_L(w,z) %\leq C L^{3/2}T_s^{3/2} \abs{w-z}
    & \leq 
    \bigg(
    \frac{ \Gamma\big(\frac{\alpha+3}{2}\big) }{ y_{\textrm{min}}^2} 
    \big( 
     2 +  4 (x_{\textrm{max}}^2 + y_{\textrm{max}}^2) 
    \big)^{\frac{1}{2}}
    \notag \\*
    & \qquad 
    + \bigg(\frac{\Gamma(\alpha +2)}{2^{\alpha+2} y_{\textrm{min}}}
    +
    \frac{ \Gamma(\alpha+2)}{2^{\alpha+1}y_{\textrm{min}}}
    +  \frac{ 2 \Gamma(\alpha) }{(2y_{\textrm{min}})^{\alpha}}
    \bigg)^{\frac{1}{2}}
    \bigg) L^{ \frac{3}{ \alpha+2 }} \abs{w-z} 
  \ea
  from which the result follows immediately.
\end{proof}

\begin{lemma}\label{lem:psimonoton}
  Let $\Psi(x) = \frac{1}{\sqrt{2\pi}}\int_x^{\infty} e^{-\frac{t^2}{2}} \mathrm{d}t$. Then for any $a>0$, $m>0$, and $k>0$ the function
  \be
    h(x) = x^{k}\Psi(a x^m)
  \ee
  is monotonically decreasing for all $x^{m } \geq \frac{k\sqrt{2\pi}}{2 a m}$.
\end{lemma}
\begin{proof}
  We have  
  \be
      h'(x)=  k x^{k-1} \Psi(a x^m) - \frac{1}{\sqrt{2\pi}} x^{k} e^{-\frac{(ax^m)^2}{2}} a m x^{m-1}\,.
  \ee
  Thus, it remains to show that 
  \be
     \frac{1}{\sqrt{2\pi}} x^{k} e^{-\frac{(ax^m)^2}{2}} a m x^{m-1} >  k x^{k-1} \Psi(a x^m)
  \ee
  which is equivalent to
  \be
    \frac{a m}{k} x^{m } e^{-\frac{(ax^m)^2}{2}}  > \int_{a x^m}^{\infty} e^{-\frac{t^2}{2}} \mathrm{d}t\,.
  \ee
  Substitution results in 
  \be
    \int_{a x^m}^{\infty} e^{-\frac{t^2}{2}} \mathrm{d}t 
      = \int_{0}^{\infty} e^{-\frac{(s+a x^m)^2}{2}} \mathrm{d}s
      \leq  e^{-\frac{(ax^m)^2}{2}} \int_{0}^{\infty} e^{-\frac{s^2}{2}} \mathrm{d}s
  \ee
  and thus it is sufficient to have $\int_{0}^{\infty} e^{-\frac{s^2}{2}} \mathrm{d}s \leq \frac{a m}{k} x^{m }$ concluding the proof.
\end{proof}

 % \bibliographystyle{IEEEtran}

\end{document}